\documentclass[final]{siamltex}
\usepackage{cite}
\usepackage{amsxtra,amssymb,amsfonts,amsmath}
\usepackage{enumerate}

\def\ds{\displaystyle}

\def\R{\mathbb{R}}

\title{Optimal control of a thermistor problem with vanishing conductivity} %\thanks{This
\author{Volodymyr Hrynkiv\thanks{Department of Mathematics and Statistics, University of Houston - Downtown,
        Houston, TX 77002 USA ({\tt HrynkivV@uhd.edu}).}
        \and Sergiy Koshkin\thanks{Department of Mathematics and Statistics, University of Houston - Downtown,
        Houston, TX 77002 USA ({\tt koshkins@uhd.edu}).}
        }

\begin{document}

\maketitle

\begin{abstract}
An optimal control of a steady state thermistor problem is considered, where the convective boundary coefficient is taken as the control variable. A distinctive feature of this paper is that
the problem is considered in arbitrary dimensions, and the electrical conductivity is allowed to vanish above a threshold temperature value. The existence of a steady state is
proved, an objective functional is introduced, the existence of the optimal control is proved, and the optimality system is derived. 
\end{abstract}



\pagestyle{myheadings} \thispagestyle{plain} \markboth{V. HRYNKIV,
S. KOSHKIN}{OPTIMAL CONTROL OF A THERMISTOR PROBLEM}

\section{Introduction}\label{Intro}
Thermistor is a device whose electrical conductivity is highly
sensitive to temperature, namely, its electrical conductivity may
change by several orders of magnitude with the increase of
temperature. Thermistors are often used as temperature control
elements in a wide variety of military and industrial equipment
ranging from space vehicles to air conditioning controllers. They
are also used in the medical field for localized and general body
temperature measurement, in meteorology for weather forecasting,
and in chemical industries as process temperature sensors
\cite{Kwo,Mac}. A detailed description of thermistors and their
applications in electronics and other industries can be found in
\cite{Mac}.

We consider the following steady-state thermistor problem
\begin{eqnarray}\label{Prob1.1}
{\ds \nabla\cdot(\sigma(u)\nabla\varphi) }&=& {0  \quad {\rm in}\
\Omega, }\nonumber\\
{\ds \Delta u+{\sigma}(u)|\nabla\varphi|^{2} }&=&{0  \quad {\rm
in}\ \Omega,} \nonumber\\[-1.5ex]
\label{statet} \\[-1.8ex]
{\ds \frac{\partial u}{\partial n}+\beta (u-u_1)}&=&{0  \quad {\rm
on}\ \Gamma_R,} \nonumber\\
 {\ds u}&=&{u_0  \quad {\rm on}\
 \Gamma_D,}\nonumber\\
{\ds \varphi}&=&{\varphi_0  \quad {\rm on}\
 \partial\Omega,}\nonumber
\end{eqnarray}
where $\varphi( x)$ is the electric potential, $u(x)$ is the
temperature, $\sigma(u)$ is the electrical conductivity, and $n$
denotes the outward unit normal. %The first equation represents the
%conservation of charge and the second equation describes the
%steady diffusion of heat in the presence of Joule heating due to
%the electric current. Boundary conditions show how the thermistor
%is connected thermally and electrically to its surroundings.
A more detailed discussion of the physical justification of
equations (\ref{statet}) can be found in \cite{Fow,How,Shi}.

The heat transfer coefficient $\beta$ is taken to be a control
variable and the set of admissible controls is denoted by $U_{\cal
M}:=\{\beta\in L^\infty(\Gamma_R):\,0\leq\beta\leq {\cal M}\}$.
Assumptions on the data are as follows:
\begin{remunerate}
\item $\Omega\subset{\R}^{d}$, $d>2$, is a bounded domain with
piecewise smooth (at least $C^2$) boundary $\partial\Omega$, where
$\partial\Omega={\overline{\Gamma}_D}\cup{\overline{\Gamma}_R}$
with $\Gamma_D\neq \emptyset$.

\item $\sigma(s)\in C^1([0,\infty))$, $\sigma(s)$ is a positive function,
monotone decreasing on $[0, u_{*})$, where $u_*$ is a threshold
temperature, that is, $\sigma(u)=0$ for $u\geq u_*$ and
$\sigma(u)>0$ for $u<u_*$. In addition,
$\|\sigma\|_{C^{1}([0,u_{*}])}\leq \mu<\infty$ and $\int_0^{u_{*}}
\frac{ds}{\sigma(s)}=\infty$. The value $u_{*}$ is called the {\it
critical temperature}, where we allow $u_{*}=\infty$, but in
realistic thermistors $u_{*}<\infty$.

\item Extending $\varphi_0$ to the whole domain $\Omega$ we assume that
$\varphi_0\in C^1(\overline{\Omega})$, moreover $\|\varphi_0\|_{W^{1,\infty}(\Omega)}$ is sufficiently small.

\item $u_0,u_1\geq 0$ a.e., $\beta\in U_{\cal M}$, and also extending to the whole domain $u_0, u_1\in
C^1(\overline{\Omega})$. Moreover, $\|u_0\|_{L^{\infty}(\Omega)},\,\|u_1\|_{L^{\infty}(\Omega)}<u_*$,
i.e. the boundary data is bounded away from the critical temperature.
\end{remunerate}

The first paper on an optimal control problem for thermistor was a
time dependent one considered by Lee and Shilkin in \cite{Lee},
where the source term was taken to be the control. The most recent
paper on optimal control of thermistor equations is by
Meinlschmidt, Meyer, and Rehberg \cite{Mei}, where the time
dependent problem from \cite{Hom} is extended to three dimensions
and where the control is taken to be the current on a part of the
boundary. Ammi and Torres in \cite{Amm} considered an optimal
control of nonlocal thermistor equations.

Our paper is a generalization of \cite{Hry}, and to our knowledge this is the
first paper on optimal control of thermistor with vanishing
conductivity. We generalize \cite{Hry} in several ways. First of
all, as physics demands the electrical conductivity $\sigma(u)$ is
no longer assumed to be uniformly positive. In real thermistors
the conductivity drops sharply by several orders of magnitude at
some critical temperature, and remains essentially zero for larger
temperatures, this feature is essential for the intended
functioning of thermistors as thermoelectric switches. The
restriction to two-dimensional domains is also removed, as well as
the artificial strict positivity assumption on the boundary heat
transfer coefficient $\beta$ that serves as the control.

The price for removing these restrictions is a somewhat higher
regularity imposed on the equilibrium temperature in the Robin
condition, and introduction of the Dirichlet boundary condition
for the temperature on part of the boundary. Let us explain the
reasons. 

The proof of existence of optimal control in \cite{Hry}
was based on the existence result of \cite{How}, which assumed
$\sigma(u)$ to be uniformly positive. Without this assumption the
first equation in \eqref{Prob1.1} is not uniformly elliptic, so
there may be no a priori bound on $\nabla\varphi$, and
$\sigma(u)\nabla\varphi$ may not make sense. One way to deal with
this issue is to switch to a notion of the capacity solution developed by Xu \cite{Xu0, Xu1, Xu2}.
Unfortunately, his theory is restricted to the Dirichlet/Neumann
boundary conditions, and does not provide enough regularity to
construct the optimal control. 

We adopt instead the approach developed by Chen \cite{Che} for the thermistor problem with the
Robin condition on part of the boundary (actually, Chen's
condition is even more general). It uses a modified Diesselh\"orst
substitution to transform \eqref{Prob1.1} into a similar problem
where the analog of $\sigma(u)$ no longer vanishes. The boundary
condition becomes non-linear, and the resulting equation for the
analog of $u$ is ``not very good", as Chen put it (one has to work
with a differential inequality), but after some technical labor
one is able to derive $L^\infty$ bounds on $u$ that bound it away
from the critical temperature. The Chen's original work states
the assumptions about the data only in terms of the transformed
problem, and focuses on the assumptions that do not obtain for
$\sigma(u)$ of interest to us. So our contribution in this part is
to spell out the assumptions on the original data (using Lemma
\ref{lemma1}), and to somewhat rework the proof.

To reap the fruits of our labor, however, we need a higher regularity
for $u$ than $L^\infty$, so that we can get $\varphi$ which is
better than $H^1$ for constructing the control. The requisite
improvement to $W^{1,p}$, with any $p>2$, based on the first
equation in \eqref{Prob1.1}, is provided by Theorem 7.2(iii) of
\cite{Rod}, but only assuming that $u$ (and hence $\sigma(u)$) is
in $C^0$. Chen's paper does provide the bootstrap to
H\"older continuity for $u$ (and $\varphi$) once the $L^\infty$
bound is in place, but only if $C^1$ regularity is assumed of the boundary data.
The restriction to two dimensions is removed because we no longer depend on the
Meyers estimate relied upon by \cite{How} to handle the boundary
data with weaker regularity. We believe that to accomodate
discontinuous boundary data one would have to extend Xu's capacity
solutions to the Robin problem, and develop non-smooth methods for
constructing controls in this context.

In \cite{Hry} the Robin condition on $u$ was imposed along the
entire boundary, and the heat transfer coefficient $\beta$, which
serves as the control, was bounded away from zero. Unfortunately,
this does not provide enough coercivity to guarantee that $u$
stays away from the critical temperature. To keep it from going
critical one needs to know explicitly that $u<u_*$ at least
somewhere in the domain, this is the reason for imposing the
Dirichlet condition on part of the boundary with
$\|u_0\|_{L^{\infty}(\Omega)}<u_*$. Once this is done, however, the
uniform positivity of $\beta$ is no longer necessary on the Robin
part of the boundary.

In this paper we choose the same objective functional as in
\cite{Hry}. The physical considerations leading to this objective
functional can be found in \cite{Hry} as well. Thus we have  %It is known that large temperature gradients may cause
%the thermistor to crack. It was shown (see \cite{Fow, Zho}) that
%low values of the heat transfer coefficient $\beta$ will lead to
%small temperature variations. On the other hand, since low values
%of the heat transfer coefficient lead to high operating
%temperatures of a thermistor, we take the heat transfer
%coefficient as a control and consider the optimal control problem
%of minimizing the heat transfer coefficient while keeping the
%operating temperature of the thermistor reasonably low. These
%physical considerations lead us to the following objective
%functional
$$
J(\beta)=\int_\Omega
u\,dx+\int_{\Gamma_R}\beta^2\,ds,\nonumber
$$
%Denoting the set of admissible controls by
%$$
%U_M=\{\beta\in L^\infty(\partial\Omega):\,0<\lambda\leq\beta\leq
%M\},
%$$
and the optimal control problem is:
\begin{equation}
{\rm Find}\ \beta^*\in U_{\cal M}\ {\rm such}\ {\rm that}\
J(\beta^*)=\min_{\beta\in U_{\cal M}}J(\beta).\label{OC}
\end{equation}
Henceforth we use the standard notation for Sobolev spaces, we
denote $\|\cdot\|_p=\|\cdot\|_{L^p(\Omega)}$ for each $p\in
[1,\infty]$; other norms will be explicitly labeled.

%In section \ref{apriori} we derive {\it a priori} estimates under
%the assumption of small boundary data. In section \ref{exist} we
%prove existence of an optimal control. Also, in section
%\ref{exist} we explain why the space dimension is restricted to
%$N=2$. The optimality system is derived and an optimal control is
%characterized in section \ref{optimality}. Uniqueness of the
%optimal control is proven in section \ref{uniqueness}.

\section{Existence of a weak solution and its regularity}\label{criticaltemp}

First, let us define weak solution to \eqref{statet}. Denote
$V_D:=\{v\in H^1(\Omega): v=0 \,\,{\rm on }\,\,\Gamma_D\}$. A pair
$(u,\varphi)$ is called a weak solution if $u-u_0\in V_D,
\varphi-\varphi_{0}\in H_{0}^{1}(\Omega)$, and
\begin{eqnarray}\label{weaksol}
\int_\Omega\nabla u\nabla v\,dx+\int_{\Gamma_R}\beta
(u-u_1) v\,ds&=&\int_\Omega
(\varphi_{0}-\varphi)\sigma(u)\nabla\varphi\nabla
v\,dx+\int_\Omega(\sigma(u)\nabla
\varphi\nabla\varphi_{0})\,v\,dx \nonumber\\
\label{eq1}\\[-2.2ex]
\int_\Omega\sigma(u)\nabla \varphi\cdot\nabla w\,dx&=& 0
\hskip0.2in \forall\,w\in H^{1}_{0}(\Omega),\forall\,v\in
V_D(\Omega).\nonumber
\end{eqnarray}
The key existence result that we rely on is the following, the proof will be given in the next two sections.
\begin{theorem}\label{existreg}
Suppose conditions 1.-4. of the Introduction hold. Then system
\eqref{Prob1.1} has a weak solution $u,\varphi$. Moreover,
$u,\varphi\in C^\alpha(\overline{\Omega})$ for some $0<\alpha<1$
and $0\leq u\leq N<u_*$, where $N$ depends only on $\Gamma_D$,
$u_0,\varphi_0$ and $\|\sigma\|_{C^1}$, but not on $u_1$ or
$\beta$.
\end{theorem}

For the proof of the existence of optimal control we will need a
better regularity for $\varphi$ than simply $H^{1}$, which follows
from the following result (see \cite{Rod}, Theorem 7.2(iii), p.
82).
%due to N.~Meyers %\cite{Ben,Mey,Tay}.
\begin{theorem}\label{meyers}
Let $\Omega\subset{\R}^d$ be a bounded domain with
$\partial\Omega\in C^1$ and suppose that $u\in H_0^1(\Omega)$ is
the unique solution to
\begin{equation}
\int_\Omega(a_{ij}\phi_{x_i}+f_j)v_{x_j}\,dx=0,\hskip0.03in
\forall v\in H_0^1(\Omega),\nonumber
\end{equation}
where $a_{ij}\in C^0(\overline{\Omega}))$ and strictly elliptic. Then $\phi\in W_0^{1,r}(\Omega)$
for each $2<r<\infty$, whenever $f_j\in L^r(\Omega)$,
$j=1,\ldots,n$.
\end{theorem}

Applying this theorem to the first equation in \eqref{Prob1.1}
rewritten for $\varphi-\varphi_{0}$ we see that $a_{ij}=\sigma\circ u$
is even in $C^\alpha(\overline{\Omega})$, and strictly positive because $u<u_*$,
while $f_j=(\sigma\circ u)\nabla\varphi_0$ is even in $C^0(\overline{\Omega})$. Therefore, Theorem
\ref{meyers} guarantees that $\varphi-\varphi_{0}\in W_0^{1,r}(\Omega)$, and hence $\varphi\in W^{1,r}(\Omega)$, for each $2<r<\infty$.

Thus, we can always choose $r$ and then $s$ in such a way that
\begin{eqnarray}
&r>d>2, s\in\Big[1,\frac{2d}{d-2}\Big), {\rm and}\nonumber\\
\label{r}\\[-2.2ex]
&\frac{1}{s}+\frac{1}{r}=\frac{1}{2}.\nonumber
\end{eqnarray}
Namely, given $d>2$ choose $r=2(d-1)$ and $s=2(d-1)/(d-2)$.
Observe that this choice of $s$ guarantees that $s\in
[1,\frac{2d}{d-2})$ and hence $H^1(\Omega)\subset\subset
L^s(\Omega)$ for all $s\in [1,\frac{2d}{d-2})$. On the other hand,
the selection of $r>d$ will guarantee the compact embedding
$W^{1,r}(\Omega)\subset\subset C(\bar{\Omega})$. These (compact)
embeddings will be used when both existence of optimal control and
derivation of optimality system are considered.% Once $r$ and $s$
%are chosen, we find $s^{\prime}$, the conjugate of $s$, i.e.
%\begin{eqnarray}\label{conjugate}
%\frac{1}{s^\prime}+\frac{1}{s}=1.
%\end{eqnarray}

%Note that $r>2$ is chosen from the Meyers estimate, then
%$s^\prime$ and $s$ are determined and satisfy $s^\prime<2<s$.

%We will show that not only $\|u\|_{L^\infty}<\infty$ but moreover
%$0\leq u\leq \alpha < u_*$ a.e. for any weak solution.

\subsection{Diesselhorst-Chen substitution}

The proof of Theorem \ref{existreg} relies on analyzing a system
obtained from \eqref{Prob1.1} by a modification of a well-known substitution.
The usual Diesselhorst substitution \cite{Die} is
$\tilde\psi=\frac{1}{2}\varphi^2+F(u)$, where $F(u):=\int_0^u
\frac{ds}{\sigma(u)}$. We will use a modified substitution,
introduced by Chen \cite{Che}, $\psi:=(\varphi-\varphi_0)^2+F(u)$
to accommodate the Robin boundary condition. First, note that $F$
maps $[0,u_*)$ onto $[0,\infty)$ since $\int_0^u
\frac{ds}{\sigma(s)}=\infty$ by assumption, and
$F^{-1}:[0,\infty)\rightarrow [0,u*)$ is well defined. Let
$v:=\int_0^u\frac{ds}{\sigma(s)}=F(u)$, not to be confused with $v$ in \eqref{weaksol},
which was just a generic test function. We introduce
$a(v):=\sigma(F^{-1}(v))$, and hence
$\nabla v=F'(u)\nabla u=\frac{1}{\sigma(u)}\nabla u=\frac{1}{\sigma(F^{-1}(v))}
\nabla u=\frac1{a(v)}\nabla u$. Therefore $\nabla u=a(v)\nabla v$, and the original system (\ref{statet}) can be
written as:
\begin{eqnarray}\label{Chensyst}
{\ds \nabla\cdot(a(v)\nabla\varphi) }&=& {0  \quad {\rm in}\
\Omega, }\nonumber\\
{\ds \nabla\cdot(a(v)\nabla v)+{a}(v)|\nabla\varphi|^{2} }&=&{0
\quad {\rm
in}\ \Omega,} \nonumber\\[-1.5ex]
\label{statetDC} \\[-1.8ex]
{\ds \frac{\partial v}{\partial n}+\frac{\beta}{a(v)}
(F^{-1}(v)-u_1)}&=&{0 \quad {\rm
on}\ \Gamma_R,} \nonumber\\
 {\ds v}&=&{F(u_0)  \quad {\rm on}\
 \Gamma_D,}\nonumber\\
{\ds \varphi}&=&{\varphi_0  \quad {\rm on}\
 \partial\Omega.}\nonumber
\end{eqnarray}
As with the system \eqref{Prob1.1} we understand \eqref{Chensyst}
in the weak sense analogous to \eqref{weaksol}. One advantage of this new system is that now
$a(v)>0$ on $[0,\infty)$, unlike $\sigma(u)$. We can now outline the proof of Theorem \ref{existreg}
modulo the a priori $L^\infty$ estimate of Theorem \ref{Linftyest}, whose very technical proof
is postponed until the next section.

{\it Proof of Theorem \ref{existreg}.}
By Theorem \ref{Linftyest} any weak solution to \eqref{Chensyst} is in $L^\infty$, which means that the corresponding weak solution to \eqref{Prob1.1} is bounded away from the critical temperature:
\begin{eqnarray}\label{mainthmu}
0\leq u\leq F^{-1}(\|v\|_\infty)=:M<u_*\,.
\end{eqnarray}
Let $\sigma_n(s)$ be equal to $\sigma(s)$ for $s\leq n$, satisfy $\sigma_n(s)\geq\sigma(n)/2$ for all $s\geq0$, and
$\|\sigma_n\|_{C^1}=\|\sigma\|_{C^1}$. Such a $\sigma_n$ can always be produced by interpolation.

By \eqref{Linftyest} and \eqref{capC} the $L^\infty$ estimate for $\|v\|_{\infty}$, and hence the value of $N$, only depend on $\|\sigma\|_{C^1}=\mu$ and $u_0,\varphi_0$.
In particular, it is independent of $\beta$, and of $n$. But if $n<u_*$ then $\sigma_n$ is bounded away from $0$ for all $s\geq0$, so by the main result of \cite{How} system \eqref{Prob1.1}
with $\sigma$ replaced by $\sigma_n$ has a weak solution. For $n>N$ this weak solution will actually be a solution to \eqref{Prob1.1} with $\sigma$ itself, and the corresponding $v,\varphi$ will be a weak solution to \eqref{Chensyst} bounded in $L^\infty$.
Together with the $C^1$ regularity of $v_0,\varphi_0$ by Lemmas 3 and 5 of \cite{Che} $v,\varphi$ are then bootstrapped to $C^\alpha(\overline{\Omega})$ for some $\alpha\in(0,1)$.
Since $u=F^{-1}(v)$ and $F\in C^2$ under our assumptions about $\sigma$ the same is true of $u$. $\endproof$

The result of \cite{How} also guarantees uniqueness of solution when the boundary data $v_0,\varphi_0$ are sufficiently ``small". Unfortunately, this does not extend here.
Of course, there is uniqueness for each $\sigma_m$, but it is conceivable that \eqref{Prob1.1} also has weak solutions that are not bounded away from $u_*$ even for small boundary data.
They would not solve \eqref{Prob1.1} with $\sigma_m$ in place of $\sigma$ for any $m<u_*$, and the Diesselhorst-Chen substitution would not be defined for them, so there would not be a
corresponding $v$.

\subsection{$L^\infty$ estimates}

In this section we give a proof of Theorem \ref{Linftyest}, and therefore complete the proof of Theorem \ref{existreg}.
This is accomplished through a series of lemmas following the general outline of \cite{Che}. By assumption on $\sigma$, we have $v=F(u)=\int_0^u\frac{ds}{\sigma(s)}\rightarrow\infty$ as
$u\rightarrow u_*$. Therefore, if we can show that
$\|v\|_\infty<\infty$ not only will we have $\|u\|_\infty<\infty$
but also $0\leq u\leq F^{-1}(\|v\|_\infty)<u_*$ a.e., that is $u$
is bounded away from the critical temperature everywhere in the
domain.
\begin{lemma}\label{lemma1}
The function $a(v)$ is strictly positive, monotone decreasing, and
\begin{equation}\label{bona}
e^{-\mu|y|}\leq\frac{a(v+y)}{a(v)}\leq e^{\mu|y|}.
\end{equation}
Moreover, for any $p\geq 2$, we have $$\ds
{\frac{a(v)}{v^p}\int_0^v\frac{s^{p-2}}{a(s)}\,ds}\xrightarrow[v\to\infty]{}0\,.
$$
\end{lemma}

\begin{proof}
Since $F^{-1}([0,\infty))=[0,u_*)$ and $\sigma>0$ on $[0,u_*)$, it
follows that $a(v)>0$. Since $\sigma\geq 0$ the function $F$ is
monotone increasing, and therefore so is $F^{-1}$. As $\sigma$ is
monotone decreasing and $a=\sigma\circ F^{-1}$, we have the same
for $a$. To prove (\ref{bona}), consider
$$
\ln{\frac{a(v+y)}{a(v)}}=\ln{a(v+y)}-\ln{a(v)}=\int_v^{v+y}(\ln{a(s)})^\prime\,
ds=\int_v^{v+y}\frac{a^\prime(s)}{a(s)}\,ds\,,
$$
where we assumed $y\geq 0$ for definiteness. Now by the chain rule
\begin{multline}
a^\prime(s)=a^\prime(F^{-1}(s))(F^{-1}(s))^\prime=\sigma^\prime(F^{-1}(s))\frac{1}{F^\prime(F^{-1}(s))}\\
=\sigma^\prime(F^{-1}(s))\frac{1}{1/\sigma(F^{-1}(s))}=\sigma^\prime(F^{-1}(s))\,a(s).\nonumber
\end{multline}
Since $|\sigma^{\prime}|\leq\mu$ this implies $|\frac{a^\prime(s)}{a(s)}|\leq \mu$, and therefore
\begin{eqnarray}
\Big|\ln\frac{a(v+y)}{a(v)}\Big|\leq \int_v^{v+y}\mu\,ds&=&\mu|y|,
{\rm \quad or \quad equivalently}\nonumber\\
-\mu|y|\leq\ln\frac{a(v+y)}{a(v)}&\leq & \mu|y|.\nonumber
\end{eqnarray}
Exponentiating the last inequality gives (\ref{bona}). The claim
about the limit follows from the monotone decrease of $a(v)$.
Namely, we have
\begin{eqnarray}
\frac{a(v)}{v^p}\int_0^v\frac{s^{p-2}}{a(s)}\,ds=\frac{1}{v^p}\int_0^v
s^{p-2}\frac{a(v)}{a(s)}\,ds\leq\frac{1}{v^p}\int_0^v v^{p-2}\cdot
1\, ds=\frac{1}{v}\xrightarrow[v\to\infty]{}0\,. \nonumber
\end{eqnarray}
\end{proof}

Due to non-vanishing of $a(v)$ we immediately have from the
maximum principle that for any weak solution to (\ref{statetDC}),
$v\geq 0$ and $\|\varphi\|_\infty\leq \|\varphi_0\|_\infty$. Note
that we could not infer this from (\ref{statet}) directly
since $\sigma(u)$ may be vanishing. The hard part is to obtain an
$L^\infty$ estimate for $v$, and therefore, for $u=F^{-1}(v)$. To
this end, we set $\psi:=(\varphi-\varphi_0)^2+v$. Then
$\nabla\psi=2(\varphi-\varphi_0)\nabla(\varphi-\varphi_0)+\nabla
v$, and $\frac{\partial\psi}{\partial n}=\frac{\partial
v}{\partial n}, \psi=v$ on $\partial\Omega$ since
$\varphi=\varphi_0$ on $\partial\Omega$. Although we write differential identities for simplicity here and below they should be understood in the weak form as in \eqref{weaksol}.

\begin{lemma}\label{lemma2}
For a weak solution $(v,\varphi)$ to (\ref{statetDC}) and
$\psi:=(\varphi-\varphi_0)^2+v$, one has
\begin{eqnarray}
\nabla\cdot(a(v)\nabla\psi)&=&a(v)|\nabla\varphi|^2-2\nabla\cdot((\varphi-\varphi_0)a(v)\nabla\varphi_0)-2a(v)\nabla\varphi\cdot\nabla\varphi_0.\nonumber
\end{eqnarray}
Moreover,
\begin{eqnarray}
-\nabla\cdot(a(v)\nabla\psi)&\leq&a(v)|\nabla\varphi_0|^2+2\nabla\cdot((\varphi-\varphi_0)a(v)\nabla\varphi_0).\nonumber
\end{eqnarray}
\end{lemma}

\begin{proof}
We write $a$ instead of $a(v)$ for short. Since
$\nabla\psi=2(\varphi-\varphi_0)\nabla(\varphi-\varphi_0)+\nabla
v$, by direct computation we obtain:
$$
\nabla\cdot(a\nabla\psi)=\nabla\cdot(a\nabla
v)+2\nabla\cdot(a(\varphi-\varphi_0)\nabla\varphi)\\-2\nabla\cdot(a(\varphi-\varphi_0)\nabla\varphi_0)\,.
$$
Simplifying the middle term, we have
$$
2\nabla\cdot(a(\varphi-\varphi_0)\nabla\varphi)=2(\varphi-\varphi_0)\nabla\cdot(a\nabla\varphi)+
2a\nabla\varphi\cdot\nabla(\varphi-\varphi_0)=0+2a|\nabla\varphi|^2-2a\nabla\varphi\cdot\nabla\varphi_0
$$
by the first equation in (\ref{statetDC}). By the second equation in (\ref{statetDC}):
$\nabla\cdot(a\nabla v)=-a|\nabla\varphi|^2$, so that
$$
\nabla\cdot(a\nabla\psi)=-a|\nabla\varphi|^2+2a|\nabla\varphi|^2-2a\nabla\varphi\cdot\nabla\varphi_0-2\nabla\cdot((\varphi-\varphi_0)a\nabla\varphi_0),
$$
which yields the desired equation. To obtain the inequality, note
that $-a|\nabla\varphi|^2+2a\nabla\varphi\cdot\nabla\varphi_0\leq
a|\nabla\varphi_0|^2$ by applying $2xy\leq x^2+y^2$ with
$x=a^{1/2}\nabla\varphi$ and $y=a^{1/2}\nabla\varphi_0$.
\end{proof}

Since $\psi=(\varphi-\varphi_0)^2+v$ and we already know
that $\|\varphi\|_\infty\leq \|\varphi_0\|_\infty$ it suffices to
show that $\|\psi\|_\infty<\infty$. In view of Lemma~\ref{lemma2} we will be working with the
inequality
\begin{eqnarray}
-\nabla\cdot(a(v)\nabla\psi)&\leq&a(v)|\nabla\varphi_0|^2+2\nabla\cdot(a(v)(\varphi-\varphi_0)\nabla\varphi_0),\nonumber\\
\psi&=&F(u_0) {\quad\rm on \quad} \Gamma_D,\nonumber\\
[-1.5ex]
\label{vi} \\[-1.8ex]
\frac{\partial\psi}{\partial
n}&+&\frac{\beta}{a(v)}(F^{-1}(v)-u_1)=0 {\quad\rm on \quad}
\Gamma_R.\nonumber
\end{eqnarray}

Following Chen's suggestion in \cite{Che} we set
$\xi(\psi):=\int_M^\psi\frac{s^{p-2}}{a(s)}\, ds$ (where $M$ is a
positive constant to be specified later), multiply both sides of
(\ref{vi}) by $\xi$, and integrate by parts. We set
$\psi_M:=\max\{M,\psi\}$.

\begin{lemma}\label{lemma3}
For any $p\geq 2$ and $M>F(\|u_0\|_\infty)$, the following
estimate holds
\begin{multline}\label{vvi}
\int_\Omega\frac{a(v)}{a(\psi_M)}\psi_M^{p-2}|\nabla\psi_M|^2+\int_{\Gamma_M}\beta\xi(\psi_M)(F^{-1}(v)-u_1)\\
\leq\int_\Omega\frac{a(v)}{a(\psi_M)}|\nabla\varphi_0|^2
a(\psi_M)\xi(\psi_M)-2\int_{\Omega}\frac{a(v)}{a(\psi)}(\varphi-\varphi_0)\psi_M^{p-2}\nabla\varphi_0\nabla\psi_M\,,
\end{multline}
where $\Gamma_M:=\partial\Omega\cap\{\psi>M\}$
\end{lemma}

\begin{proof}
By the product rule,
$$\nabla\cdot(a(v)\nabla\psi)\xi(\psi_M)=\nabla\cdot(a(v)\xi(\psi_M)\nabla\psi)-a(v)\nabla\psi\cdot\nabla\xi(\psi_M),$$
and by definition of $\xi$, we have
$\nabla\xi(\psi_M)=\frac{\psi_M^{p-2}}{a(\psi_M)}\nabla\psi_M$.
Therefore, $$-\int_\Omega
\nabla\cdot(a(v)\nabla\psi)\xi(\psi_M)=\int_\Omega\frac{a(v)}{a(\psi_M)}\psi_M^{p-2}\nabla\psi\cdot\nabla\psi_M-\int_{\partial\Omega}a(v)\xi(\psi_M)\frac{\partial\psi}{\partial
n}.
$$
We now stipulate that $M>F(\|u_0\|_\infty)$. Then
$\psi|_{\Gamma_D}<M$ and $\psi_M|_{\Gamma_D}=M$,\\
$\xi(\psi_M)|_{\Gamma_D}=0$. Hence, the boundary integral reduces
to $\Gamma_R$, where $\psi=\psi_M$ and
$\frac{\partial\psi}{\partial
n}=-\frac{\beta}{a(v)}(F^{-1}(v)-u_1)$. We can also replace $\psi$
by $\psi_M$ in the interior integrals because $\xi(\psi_M)=0$ on
the set where $\psi<\psi_M$. This leads to
$$-\int_\Omega
\nabla\cdot(a(v)\nabla\psi_M)\xi(\psi_M)=\int_\Omega\frac{a(v)}{a(\psi_M)}\psi_M^{p-2}|\nabla\psi_M|^2+\int_{\Gamma_M}\beta\xi(\psi_M)(F^{-1}(v)-u_1),
$$
where $\Gamma_M\subset\Gamma_R$ is the part of the boundary where
$\psi>M$. Similarly,
$$\nabla(a(v)(\varphi-\varphi_0)\xi(\psi_M)\nabla\varphi_0)=\nabla(a(v)(\varphi-\varphi_0)\nabla\varphi_0)\xi(\psi_M)+
a(v)(\varphi-\varphi_0)\frac{\nabla\varphi_0\psi_M^{p-2}}{a(\psi_M)}\nabla\psi_M,
$$ and
\begin{eqnarray}
\int_\Omega\nabla(a(v)(\varphi-\varphi_0)\nabla\varphi_0)\xi(\psi_M)&=&\int_{\partial\Omega}a(v)(\varphi-\varphi_0)\xi(\psi_M)\frac{\partial\varphi_0}{\partial
n}\nonumber\\
&-&\int_{\Omega}\frac{a(v)}{a(\psi_M)}(\varphi-\varphi_0)\nabla\varphi_0\psi_M^{p-2}\nabla\psi_M,\nonumber
\end{eqnarray}
where the boundary term vanishes because
$\varphi|_{\partial\Omega}=0$. Since $\xi(\psi_M)\geq 0$ the
inequality (\ref{vi}) yields (\ref{vvi}).
\end{proof}

We now convert (\ref{vvi}) into an estimate that will be used for
bootstrapping $\psi_M$ (and hence $\psi$) into $L^\infty$. This
involves further increase for $M$.

\begin{lemma}\label{lemma4}
Let $p\geq 2$ and $M>4\|\varphi_0\|_\infty^2+F(\|u_0\|_\infty)$.
Then the following estimate holds
\begin{eqnarray}\label{vvvi}
\int_\Omega\psi_M^{p-2}|\nabla\psi_M|^2\leq\frac{8e^{8\mu
\|\varphi_0\|_{\infty}^2}\|\varphi_0\|_{W^{1,\infty}}^2}{1-2\varepsilon
e^{8\mu
\|\varphi_0\|_{\infty}^2}\|\varphi_0\|_{W^{1,\infty}}^2}\int_\Omega(\varepsilon\psi_M^p+C_{\varepsilon}+\varepsilon^{1-p}),
\end{eqnarray}
where $\varepsilon$ and $C_{\varepsilon}$ are chosen so that
$a(v)\int_0^v\frac{s^{p-2}}{a(s)}\,ds\leq\varepsilon
v^{p}+C_{\varepsilon}$ for all $v>0$.
\end{lemma}

\begin{proof}
Since $\psi>M$ on $\Gamma_M$, we have
$v=\psi-(\varphi-\varphi_0)^2>\psi-4\|\varphi_0\|_\infty^2>M-4\|\varphi_0\|_\infty^2>F(\|u_1\|_\infty)$,
so $F^{-1}(v)>u_1$ a.e. on $\Gamma_M$, meaning that the boundary
integral in (\ref{vvi}) is positive. Moreover, $|\psi-v|\leq
4\|\varphi_0\|_\infty^2$, so on the part of $\Omega$ where
$\xi(\psi_M)\neq 0$, we have
$e^{-4\mu\|\varphi_0\|_\infty^2}\leq\frac{a(v)}{a(\psi_M)}\leq
e^{4\mu\|\varphi_0\|_\infty^2}$ by Lemma~\ref{lemma1}. In view of
this, (\ref{vvi}) implies
\begin{eqnarray}
e^{-4\mu\|\varphi_0\|_\infty^2}\int_\Omega\psi_M^{p-2}|\nabla\psi_M|^2&\leq&
e^{4\mu\|\varphi_0\|_\infty^2}\Big(\|\varphi_0\|_{\infty}^2\int
a(\psi_M)\xi(\psi_M)\nonumber\\
&+&2(2\|\varphi_0\|_{\infty}^2)\|\nabla\varphi_0\|_{\infty}\int_\Omega\psi_M^{p-2}|\nabla\psi_M|\Big),\nonumber
\end{eqnarray}
and
\begin{eqnarray}
\int_\Omega\psi_M^{p-2}|\nabla\psi_M|^2&\leq&
4e^{8\mu\|\varphi_0\|_\infty^2}\|\varphi_0\|_{W^{1,\infty}}^2
\Big(\int_\Omega
a(\psi_M)\xi(\psi_M)+\int_\Omega\psi_M^{p-2}|\nabla\psi_M|\Big).\nonumber
\end{eqnarray}
By Lemma~\ref{lemma1},
$$\frac{a(v)\xi(v)}{v^p}=\frac{a(v)}{v^p}\int_M^v\frac{s^{p-2}}{a(s)}\,ds\leq\frac{a(v)}{v^p}\int_0^v\frac{s^{p-2}}{a(s)}\,ds\xrightarrow[v\rightarrow\infty]{}0\,.
$$
Therefore, $a(v)\xi(v)\leq\varepsilon
v^p+C_\varepsilon$ for arbitrarily small $\varepsilon$ and
suitable $C_\varepsilon$. Hence, $\int_\Omega
a(\psi_M)\xi(\psi_M)\leq\int_\Omega(\varepsilon\psi_M^p+C_\varepsilon)$.

By the Cauchy inequality with $\varepsilon$, $xy\leq
\frac{1}{2\varepsilon}x^2+\frac{\varepsilon}{2}y^2$, applied to
$x=\psi_M^{(p-2)/2}$ and $y=\psi_M^{(p-2)/2}|\nabla\psi|$, we have
$$
\psi_M^{p-2}|\nabla\psi|\leq\frac{1}{2\varepsilon}\psi_M^{p-2}+\frac{\varepsilon}{2}\psi_M^{p-2}|\nabla\psi_M|^2.
$$
Finally, applying the Young inequality, $xy\leq
\frac{x^{\alpha}}{\alpha}+\frac{x^{\beta}}{\beta}$
for $\frac{1}{\alpha}+\frac{1}{\beta}=1$,
$\alpha,\beta>0$, with $x=\frac{1}{2\varepsilon\delta},
y=\delta|\psi_M^{p-2}|$, we obtain
$\frac{1}{2\varepsilon}\psi_M^{p-2}\leq\frac{1}{\alpha}\frac{1}{(2\varepsilon\delta)^{\alpha}}+\frac{1}{\beta}(\delta\psi_M^{p-2})^{\beta}$.
Now select $\alpha$ and $\beta$ so that
$(p-2)\beta=p$, that is, $\beta:=p/(p-2)$ and
$\alpha:=p/2$. Then we have
$$\frac{1}{2\varepsilon}\psi_M^{p-2}\leq\frac{2}{p}\frac{1}{(2\varepsilon\delta)^{p/2}}+\Big(1-\frac{2}{p}\Big)\delta^{\frac{p}{p-2}}\psi_M^p
\leq\frac{1}{(\varepsilon\delta)^{(p/2)}}+\delta^{p/(p-2)}\psi_M^p
$$
since $p\geq 2$. We now set $\delta^{p/(p-2)}=\varepsilon$ so that
$(\varepsilon\delta)^{p/2}=\varepsilon^{p/2}(\varepsilon^{(p-2)/p})^{p/2}=\varepsilon^{p-1}$,
producing
$\frac{1}{2\varepsilon}\psi_M^{p-2}\leq\varepsilon\psi_M^p+\varepsilon^{1-p}$.
Thus,
\begin{eqnarray}
\int_\Omega\psi_M^{p-2}|\nabla\psi_M|^2&\leq& 4e^{8\mu
\|\varphi_0\|_{\infty}^2}\|\varphi_0\|_{W^{1,\infty}}^2\Bigg(\int_\Omega 2\varepsilon\psi_M^p\nonumber\\
[-1.5ex]
\label{lemma4estimate} \\[-1.8ex]
&+&\int_\Omega\{\frac{\varepsilon}{2}\psi_M^{p-2}|\nabla\psi_M|^2+C_{\varepsilon}+\varepsilon^{1-p}\}\Bigg)\nonumber,
\end{eqnarray}
which is a rearragement of (\ref{vvvi}).
\end{proof}

\begin{eqnarray}\label{capC}
\text{Set }C:=\frac{8e^{8\mu
\|\varphi_0\|_{\infty}^2}\|\varphi_0\|_{W^{1,\infty}}^2}{1-2\varepsilon
e^{8\mu \|\varphi_0\|_{\infty}^2}\|\varphi_0\|_{W^{1,\infty}}^2}.
\end{eqnarray}

\begin{corollary}\label{corollary1}
The following estimates hold
\begin{eqnarray}
\|\psi_M\|_2^2&\leq& \frac{2\text{ \rm mes}(\Omega)}{1-2\varepsilon C
C_D^2}\{CC_D^2(C_\varepsilon+\varepsilon^{-1})+M^2\text{ \rm mes}(\Omega)\}\|\nabla\psi_M\|_2^2\nonumber\\
&\leq&\varepsilon
C\|\psi_M\|_2^2+C(C_\varepsilon+\varepsilon^{-1})\text{ \rm mes}(\Omega)\nonumber
\end{eqnarray}
where $C_D:=C(\Omega,\Gamma_D)$ is a constant from the
Poincar\'{e} inequality.
\end{corollary}
\begin{proof}
With $p=2$ in (\ref{vvvi}), we have
$\int_\Omega|\nabla\psi_M|^2\leq
C\int_\Omega(\varepsilon\psi_M^2+C_\varepsilon+\varepsilon^{-1})$.
Since $\psi=F(u_0)$ and $M>F(\|u_0\|_\infty)$, we have that
$\psi_M=\max\{\psi,M\}=M$ on $\Gamma_D$, and therefore,
$\psi_M-M=0$ on $\Gamma_D$. By the Poincar\'{e} inequality, there
exists $C_D:=C(\Omega,\Gamma_D)$ such that
$$
\|\psi_M-M\|_2\leq C_D\|\nabla(\psi_M-M)\|_2=C_D\|\nabla\psi_M\|_2,
$$
and therefore, $\|\psi_M\|_2\leq C_D\|\nabla\psi_M\|_2+M\text{ mes}(\Omega)$. Thus, we obtain
$$
\|\psi_M\|_2^2\leq 2C_D^2\|\nabla\psi_M\|_2^2+2M^2
\text{ mes}^2(\Omega)\leq
2CC_D^2\int_\Omega(\varepsilon\psi_M^2+C_\varepsilon+\varepsilon^{-1})+2M^2
\text{ mes}^2(\Omega).
$$
This implies
\begin{eqnarray}\nonumber
\|\psi_M\|_2^2\leq  \frac{2\text{ mes}(\Omega)}{1-2\varepsilon C
C_D^2}\{CC_D^2(C_\varepsilon+\varepsilon^{-1})+M^2\text{ mes}(\Omega)\}\|\nabla\psi_M\|_2^2,
\end{eqnarray}
and the second estimate is obvious from (\ref{vvvi}) with $p=2$.
\end{proof}

Note that in Lemma~\ref{lemma4} the choice of $\varepsilon$ does
not depend on $M$, and the estimates (\ref{vvvi}),
(\ref{lemma4estimate}) are valid for $\psi_M$ with any
$M>4\|\varphi_0\|_\infty^2+F(\|u_{0,1}\|_\infty)$.

\begin{corollary}\label{corollary2}
Let $M>\max\{1, C_\varepsilon, 1/\varepsilon,
4\|\varphi_0\|_\infty^2+F(\|u_{0,1}\|_\infty)\}$. Then, for any
$p\geq 2$, we have
\begin{eqnarray}\label{cor2est}
\int_\Omega\psi_M^{p-2}|\nabla\psi_M|^2\leq
C(M+\varepsilon)\int_\Omega\psi_M^p.
\end{eqnarray}
\end{corollary}
\begin{proof}
With our choice of $M$,
$C_\varepsilon+\varepsilon^{1-p}=C_\varepsilon+\varepsilon\frac{1}{\varepsilon^p}\leq
2M^p\leq2\psi_M^p$. Substituting this in (\ref{lemma4estimate}) yields the claim.
\end{proof}

Now, we are ready for the main theorem of this section.
\begin{theorem}\label{Linftyest}
Let $(v,\varphi)$ be a solution to (\ref{statetDC}). Then
\begin{eqnarray}
\|v\|_\infty&\leq& 4\|\varphi_0\|_\infty^2+\max\{M,
\|\psi_M\|_\infty\},\nonumber\\
[-1.5ex]
\label{mainthm}\\[-1.8ex]
\|\psi_M\|_\infty&\leq& \|\psi_M\|_{_2} \cdot\left\{
     \begin{array}{lr}
       2C_2, d\leq 2, \\
       (2C_2)^{d/2}(\frac{d}{d-2})^{d(d-2)/4}, d>2,
     \end{array}
   \right.\nonumber
\end{eqnarray}
where $C_2:=\frac{1}{2}C_1(1+C^{1/2}(M+\varepsilon)^{1/2})$, $C_1$
is the constant from a Sobolev embedding theorem for $(\Omega,\Gamma_D)$, and where
$\varepsilon$, $C$, and $M$ are defined in Lemma~\ref{lemma4},
(\ref{capC}), and Corollary~\ref{corollary2}, respectively, and
$\|\psi_M\|_2$ is estimated in Corollary~\ref{corollary1}.
In particular, $v,\psi$ satisfy a priori $L^\infty$ estimates.
\end{theorem}
\begin{proof}
Since $v=-(\varphi-\varphi_0)^2+\psi$ and $\psi_M=\max\{\psi,
M\}$ we clearly have $\|v\|_\infty\leq \max\{M,
\|\psi_M\|_\infty\}+4\|\varphi_0\|_\infty^2$, and it remains to
estimate $\|\psi_M\|_\infty$. It follows from a Sobolev embedding
theorem for $q\leq 2d/(d-2)$ when $d>2$, and any $1\leq q<\infty$
when $d\leq 2$, that $\|\phi\|_q\leq
C_1\|\phi\|_{H^1}=C_1(\|\phi\|_2+\|\nabla\phi\|_2)$. Let us apply
this to $\phi=\psi_M^{p/2}$, noting that
$\nabla\phi=(p/2)\psi_M^{(p/2)-1}\nabla\psi_M$, and
$\|\nabla\phi\|_2^2=(p^2/4)\int_\Omega\psi_M^{p-2}|\nabla\psi_M|^2\leq
(p^2/4)C(M+\varepsilon)\|\psi_M\|_p^p$ by
Corollary~\ref{corollary2}. Similarly,
$\|\phi\|_2^2=\|\psi_M\|_p^p$, and
$\|\phi\|_q=(\int_\Omega|\psi_M^{p/2}|^q)^{1/q}=(\int_\Omega\psi_M^{m
p})^{1/2m}=\|\psi_M\|_{m p}^{p/2}$, where $m:=d/(d-2)=q/2$ for
$d>2$ and any $q\in(1,\infty)$ for $d\leq 2$. We have
\begin{eqnarray}
\|\psi_M\|_{m p}^{p/2}&\leq&
C_1(\|\psi_M\|_p^{p/2}+(p/2)C^{1/2}(M+\varepsilon)^{1/2}\|\psi_M\|_p^{p/2})\nonumber\\
&\leq&C_1(1+(p/2)C^{1/2}(M+\varepsilon)^{1/2})\|\psi_M\|_p^{p/2}\nonumber
\end{eqnarray}
and since $p\geq 2$
\begin{eqnarray}\nonumber
\|\psi_M\|_{m p}^{p/2}\leq
(\frac{1}{2}C_1(1+C^{1/2}(M+\varepsilon)^{1/2})p)^{2/p}\|\psi_M\|_p=(C_2p)^{2/p}\|\psi_M\|_p
\end{eqnarray}
where we introduced
$C_2:=\frac{1}{2}C_1(1+C^{1/2}(M+\varepsilon)^{1/2})$. Since
$m>1$, this sets up a bootstrap starting with $p=2$ and proceeding
through $2m, 2m^2, 2m^3, \ldots, 2m^k \rightarrow\infty$ as
$k\rightarrow\infty$. In particular,
\begin{eqnarray}
\|\psi_M\|_{2m^{k} }&\leq& (2C_2m^{k-1})^{2/2m^{k-1}}\ldots (2C_2)^{2/2}\|\psi_M\|_2\nonumber\\
&=&(2C_2)^{1+\frac{1}{m}+\frac{1}{m^2}+\ldots+\frac{1}{m^{k-1}}}m^{\frac{1}{m}+\frac{2}{m^2}+\ldots+\frac{k-1}{m^{k-1}}}\|\psi_M\|_2.\nonumber
\end{eqnarray}
But $1+\frac{1}{m}+\frac{1}{m^2}+\ldots+\frac{1}{m^{k-1}}\leq
\frac{1}{1-\frac{1}{m}}$ and
$\frac{1}{m}+\frac{2}{m^2}+\ldots+\frac{k-1}{m^{k-1}}=\frac{1}{m}\frac{1}{(1-\frac{1}{m})^2}$,
so that
\begin{eqnarray}\label{mkest}
\|\psi_M\|_{2m^{k} }\leq
(2C_2)^{m/(m-1)}m^{m/(m-1)^2}\|\psi_M\|_2.
\end{eqnarray}
Now letting $k\rightarrow\infty$ in (\ref{mkest}) we obtain
$\|\psi_M\|_\infty$ on the left, then setting $m=d/(d-2)$ for
$d>2$, and $m\rightarrow\infty$ for $d\leq 2$, will produce
\begin{eqnarray}\nonumber
\|\psi_M\|_\infty\leq \left\{
     \begin{array}{lr}
       2C_2 \|\psi_M\|_2, d\leq 2, \\
       (2C_2)^{d/2}(\frac{d}{d-2})^{d(d-2)/4} \|\psi_M\|_2, d>2,
     \end{array}
   \right.
\end{eqnarray}
as desired. This concludes our proof of the $L^{\infty}$ estimate for $\psi$
since $\psi=(\varphi-\varphi_0)^2+F(u)$, and we already have one
for $\varphi$.
\end{proof}

\section{Further a priori estimates}\label{apriori}

%Existence of solution to (\ref{eq1}) was proven by Howison,
%Rodrigues,  and Shillor in \cite{How} by using Schauder's fixed
%point theorem. It is also shown in \cite{How} that the solution is
%unique provided the boundary data are sufficiently ``small".
%Namely, $\lambda>0$ implies (see \cite{Atk,Zei4})
%\begin{equation}\label{normequivalencet}
%\hat\kappa\|u\|_s^2\leq\int_{\Omega}|\nabla
%v|^2\,dx+\lambda\int_{\partial\Omega}v^2\,ds.
%\end{equation}
%Assume that $(u,\varphi)$ is a solution to the state system
%satisfying
%\begin{equation}\label{varphit}
%\|\nabla\varphi\|_r\leq \Phi<\infty,
%\end{equation}
%where $\Phi$ depends on $\|\varphi_0\|_{W^{1,\infty}(\Omega)}$. If
%the boundary data are sufficiently small, i.e., if
%\begin{equation}\label{assumptiont}
%K\Phi^2\Big(1+\frac{2C_2}{C_1}\Big)<\hat\kappa,
%\end{equation}
%where $\hat\kappa=\hat\kappa(s)$ is the constant from
%\eqref{normequivalencet} and $\Phi$ is the constant from
%\eqref{varphit}, then the solution is unique. Note that we will
%show next that $\varphi$ satisfies \eqref{varphit} with a bound
%depending on $\|\varphi_0\|_{W^{1,\infty}(\Omega)}$.

In order to be able to prove existence of optimal control we need
to derive more {\it a priori} estimates. In what follows, given
$\beta\in U_{\cal M}$, we denote the solution to
(\ref{eq1}) by $u(\beta)$ and $\varphi(\beta)$. %Note that as a
%result of the maximum principle solutions to (\ref{statet})
%satisfy $u\geq 0$ on $\bar\Omega$.
\begin{theorem}\label{estimates} Let $\beta\in U_{\cal M}$ be given. Then
$u(\beta)$ and $\varphi(\beta)$ solving (\ref{eq1}) satisfy
\begin{eqnarray}
\|\varphi\|_{W^{1,r}(\Omega)}&\leq &\Phi\ {\rm for}\ {\rm some}\ r>2, {\rm and}\nonumber\\
\label{ubound}\\[-2.2ex]
\|u\|_{H^1(\Omega)}&\leq &\tilde C,\nonumber
\end{eqnarray}
where $\Phi$ and $\tilde C$ are some positive constants.
\end{theorem}

\begin{proof}
%Proof of the estimate for $\|\varphi\|_{W^{1,r}(\Omega)}$ can be
%carried over directly from that of Theorem 2.2 in \cite{Hry}.
%We show the proof for $\|u\|_{H^1(\Omega)}$. From the weak
%formulation (\ref{weaksol}), we can write
First, we show the estimate for $\varphi$. We are given the
solution of a nonhomogeneous Dirichlet problem
\begin{eqnarray}
\nabla\cdot(\sigma(u)\nabla\varphi)&=&0\ {\rm in}\ \Omega, \nonumber\\
\label{eq4}\\[-2.2ex]
\varphi &=&\varphi_0\ {\rm on}\ \partial\Omega.\nonumber
\end{eqnarray}

Because of the assumption 2 from the Introduction we treat
$\sigma(u)$ in (\ref{statet}) as a bounded coefficient. Consider
the following Dirichlet problem with zero boundary data
\begin{eqnarray}
-\nabla\cdot(\sigma(u)\nabla\tilde\varphi)&=&\nabla\cdot(\sigma(u)\nabla\varphi_0)\ {\rm in}\ \Omega, \nonumber \\
\label{eq5}\\[-2.2ex]
\tilde\varphi &=&0\ {\rm on}\ \partial\Omega \nonumber.
\end{eqnarray}

%Since the right hand side of (\ref{eq5}) is in $H^{-1}(\Omega)$,
By the standard theory for elliptic equations in divergence form,
it follows that there exists $\tilde\varphi=(\varphi-\varphi_0)\in
H_0^1(\Omega)$ that solves (\ref{eq5}).
%%Since for any given $1<p<\infty$, the mapping
%%div:\hskip0.02in$(L^p(\Omega))^2\rightarrow W^{-1,p}(\Omega)$
%%defined by
%%\begin{equation}\label{onto}
%%\phi\mapsto\nabla\cdot\phi
%%\end{equation}
%%is onto, it provides $W^{-1,p}(\Omega)$ space with a ``quotient"
%%norm associated with \eqref{onto}, namely
%%\begin{equation}\label{quotientnorm}
%%\|f\|_{W^{-1,p}(\Omega)}=\inf_{\nabla\cdot
%%g=f}\|g\|_{({L^p(\Omega)})^2},
%%\end{equation}
%%(see expressions (4.10) and (4.11) in \cite{Ben}). Taking into
%%account the above facts it follows that
%%$\nabla\cdot(\sigma(u)\nabla\varphi_0)\in W_0^{-1,q}(\Omega)$ for
%%any $q\in [1,\infty)$ since $\nabla\varphi_0\in
%%L^{\infty}(\Omega)$.
By Theorem \ref{meyers}, %there exists $r>2$ such that
$\tilde\varphi\in W_0^{1,r}(\Omega)$ and
\begin{eqnarray*}
\|\tilde\varphi\|_{W_0^{1,r}(\Omega)}& \leq
C\|\sigma(u)\nabla\varphi_0\|_r\leq C\mu\|\nabla\varphi_0\|_r
{\quad\rm for}\ {\rm each}\ 2<r<\infty.
\end{eqnarray*}
Since $\|\varphi\|_{W^{1,r}(\Omega)}\leq
\|\varphi-\varphi_0\|_{W^{1,r}(\Omega)}+\|\varphi_0\|_{W^{1,r}(\Omega)}$
and
$\|\varphi-\varphi_0\|_{W^{1,r}(\Omega)}\equiv\|\tilde\varphi\|_{W^{1,r}(\Omega)}
\leq C^{\prime}\|\tilde\varphi\|_{W_0^{1,r}(\Omega)}\leq
C^{\prime}C\mu\|\nabla\varphi_0\|_r$, and taking into account that
$\|\nabla\varphi_0\|_r \leq C_3\|\nabla\varphi_0\|_{\infty}$ as
well as $\|\varphi_0\|_{W^{1,r}(\Omega)}\leq
C_4\|\varphi_0\|_{W^{1,\infty}(\Omega)}$, it is easy to see that
\begin{eqnarray*}
\|\varphi\|_{W^{1,r}(\Omega)}& \leq \Phi {\quad\rm for}\ {\rm
each}\ 2<r<\infty,
\end{eqnarray*}
where $\Phi\stackrel{{\rm def}}= C^{\prime}C C_3\mu
\|\nabla\varphi_0\|_{\infty}+C_4\|\varphi_0\|_{W^{1,\infty}(\Omega)}$.
Now we derive the estimate for $\|u\|_{H^1(\Omega)}$.
%%where we have used \eqref{quotientnorm}.
%Since
%\begin{eqnarray*}
%\|\varphi\|_{W^{1,r}(\Omega)}\leq
%\|\varphi-\varphi_0\|_{W^{1,r}(\Omega)}+\|\varphi_0\|_{W^{1,r}(\Omega)}
%\end{eqnarray*}
%and
%\begin{eqnarray*}
%\|\varphi-\varphi_0\|_{W^{1,r}(\Omega)}\equiv\|\tilde\varphi\|_{W^{1,r}(\Omega)}
%&\leq C^{\prime}\|\tilde\varphi\|_{W_0^{1,r}(\Omega)}\leq
%C^{\prime}CC_2\|\nabla\varphi_0\|_r,
%\end{eqnarray*}
%we obtain
%\begin{eqnarray*}
%\|\varphi\|_{W^{1,r}(\Omega)}\leq
%C^{\prime}CC_2\|\nabla\varphi_0\|_r+\|\varphi_0\|_{W^{1,r}(\Omega)}.
%\end{eqnarray*}
%We also have
%\begin{eqnarray*}
%\|\nabla\varphi_0\|_r&\leq&
%C_3\|\nabla\varphi_0\|_{\infty}\\
%\|\varphi_0\|_{W^{1,r}(\Omega)}&\leq&
%C_4\|\varphi_0\|_{W^{1,\infty}(\Omega)}
%\end{eqnarray*}
%whence
%\begin{eqnarray}\label{eq6}
%\|\varphi\|_{W^{1,r}(\Omega)}\leq\Phi\ {\rm for}\ {\rm some}\ r>2,
%\end{eqnarray}
%where
%\begin{eqnarray}\label{eq6a}
%\Phi\stackrel{{\rm def}}=
%C^{\prime}CC_2C_3\|\nabla\varphi_0\|_{\infty}+C_4\|\varphi_0\|_{W^{1,\infty}(\Omega)}.
%\end{eqnarray}
%%Thus, (\ref{eq6}) gives the desired estimate for $\varphi$.
%We show that $\|u\|_{H^1(\Omega)}\leq \tilde{C}$. Use $u\in
%H^1(\Omega)$ as a test function in the first equation of
%From the weak formulation (\ref{eq1}) we get
From the weak formulation (\ref{weaksol}), we can write
\begin{eqnarray}
\int_\Omega(\nabla (u-u_0)+\nabla u_0)\nabla
v\,dx+\int_{\Gamma_{R}}\beta (u-u_0)v\,ds+\int_{\Gamma_{R}}\beta
(u_0-u_1)v\,ds
\nonumber\\
\label{uestimate}\\[-2.2ex]
=\int_\Omega (\varphi{_0}-\varphi)\sigma(u)\nabla\varphi\nabla
v\,dx + \int_\Omega\sigma(u)v\nabla
\varphi\nabla\varphi{_0}\,dx\nonumber
\end{eqnarray}
Substituting $v=u-u_0$ as a test function into (\ref{uestimate}),
we obtain
\begin{eqnarray}
&\,&\int_\Omega|\nabla (u-u_0)|^2 +\int_{\Gamma_{R}}\beta
(u-u_0)^2=\int_\Omega\nabla(u_0-u)\nabla u_0 \nonumber\\
\label{uestimatecontinued}\\[-2.2ex]
&+&\int_{\Gamma_{R}}\beta(u_1-u_0)(u-u_0) +\int_\Omega
(\varphi{_0}-\varphi)\sigma(u)\nabla\varphi\nabla(u-u_0) +
\int_\Omega\sigma(u)(u-u_0)\nabla
\varphi\nabla\varphi{_0}.\nonumber
\end{eqnarray}
Since $\beta(x)\geq 0$, the left hand side of
(\ref{uestimatecontinued}) can be written as
\begin{eqnarray*}
\int_\Omega|\nabla (u-u_0)|^2\leq\int_\Omega|\nabla (u-u_0)|^2
+\int_{\Gamma_{R}}\beta (u-u_0)^2.
\end{eqnarray*}
Taking into account that $\beta\leq {\cal M}$ a.e.,
$\sup_{\Omega}\varphi\leq \tilde{M}$, and the trace inequality
$\|u\|_{L^2(\Gamma_R)}\leq {M_2}\|u\|_{H^1(\Omega)}$ (since
$u-u_0=0$ on $\Gamma_D$), the right hand side of
(\ref{uestimatecontinued}) can be estimated as follows
\begin{eqnarray*}
&\,&\int_\Omega\nabla(u_0-u)\nabla u_0
+\int_{\Gamma_{R}}\beta(u_1-u_0)(u-u_0) +\int_\Omega
(\varphi{_0}-\varphi)\sigma(u)\nabla\varphi\nabla(u-u_0) \nonumber\\
&+&\int_\Omega\sigma(u)(u-u_0)\nabla \varphi\nabla\varphi{_0} \leq
\|u_0-u\|_2 \|\nabla
u_0\|_2+{\cal{M}}\sup_{\Gamma_R}|u_1-u_0|mes^{1/2}(\Gamma_R){M_2}\|u_0-u\|_{H^1}\\
&+& 2\mu\tilde{M}\|\nabla\varphi\|_2\|\nabla
(u_0-u)\|_2+\mu\|\nabla\varphi_0\|_{_{\infty}}\|u_0-u\|_2\|\nabla\varphi\|_2
\leq\tilde{C_1}\|u_0-u\|_{H^1(\Omega)}.
\end{eqnarray*}
Hence, we have
\begin{eqnarray}\label{uestimate1}
\int_\Omega|\nabla (u-u_0)|^2\,dx\leq
\tilde{C_1}\|u-u_0\|_{H^1(\Omega)}.
\end{eqnarray}
Now, using an extension of the Poincar\'{e} inequality (as given
by Theorem 5.8 in \cite{Rod}), we obtain
\begin{eqnarray}\label{normequivalence}
k\|u-u_0\|_{H^1(\Omega)}^2\leq\|\nabla(u-u_0)\|_{2}^2\leq\tilde{C_1}\|u-u_0\|_{H^1(\Omega)},
\end{eqnarray}
which gives the desired estimate for $\|u\|_{H^1(\Omega)}$.
%It can be shown (for example, see \cite{Atk,Zei4}) that the
%quantity
%\begin{eqnarray*}
%\|v\|_{*}^2\stackrel{{\rm def}}=\int_\Omega|\nabla
%v|^2\,dx+\lambda\int_ {\partial\Omega}v^2\,ds
%\end{eqnarray*}
%defines a norm on $H^1(\Omega)$ which is equivalent to
%$\|\cdot\|_{H^1(\Omega)}$ norm; for some $k>0$
%\begin{eqnarray}\label{normequivalence}
%k\|u\|_{H^1(\Omega)}^2\leq\|u\|_{*}^2\leq\tilde{C_1}\|u\|_{H^1(\Omega)},
%\end{eqnarray}
%whence $\|u\|_{H^1(\Omega)}\leq\tilde{C}$.
%(\ref{uestimate1}) and (\ref{normequivalence}) give
%\begin{equation}\label{uestimate2}
%\|u\|_{H^1(\Omega)}\leq\tilde{C}.
%\end{equation}
%where $\tilde{C}\equiv\tilde{C_1}/k$.\qquad
\end{proof}

\section{Existence of an optimal control}\label{exist} After we
obtained $r$ and $s$ from (\ref{r}), and the corresponding a
priori estimates, we are in a position to prove existence of an
optimal control.
%Having obtained {\it a priori} estimates we proceed to the proof
%of existence of an optimal control.
\begin{theorem}\label{existencet}
There exists a solution to the optimal control problem (\ref{OC}).
\end{theorem}
\begin{proof}
We follow \cite{Hry} closely.
%Using the estimate from Theorem
%\ref{estimates}, we get
%\begin{eqnarray*}
%J(\beta)=\int_\Omega
%u\,dx+\int_{\partial\Omega}\beta^2\,ds\geq\int_\Omega u\,dx\geq
%-\int_\Omega |u|\,dx\geq-\Lambda,
%\end{eqnarray*}
%where $\Lambda>0$ is some constant, it follows that
Choose a minimizing sequence $\{\beta_n\}_{n=1}^{\infty}\subset
U_{\mathcal{M}}$ such that
$$\lim_{n\rightarrow\infty}J(\beta_n)=\inf_{\beta\in U_{\mathcal{M}}}J(\beta).$$
Let $u_n=u(\beta_n)$ and $\varphi_n=\varphi(\beta_n)$ be the
corresponding solutions to
\begin{eqnarray}
\int_\Omega\nabla u_n\nabla v\,dx+\int_{\Gamma_R}\beta_n
(u_n-u_1) v\,ds&=&\int_\Omega
(\varphi_{0}-\varphi_n)\,\sigma(u_n)\nabla
\varphi_n\nabla v\,dx\nonumber\\
&+&\int_\Omega\sigma(u_n)\nabla \varphi_n\nabla\varphi_{0}\,v\,dx
\quad\forall\,v\in
V_D(\Omega)\label{minseq}\\
\int_\Omega\sigma(u_n)\nabla \varphi_n\cdot\nabla w\,dx&=&0
\hskip0.2in \varphi_{n}-\varphi_0\in H_{0}^{1}(\Omega), u_n-u_0\in
V_D, \forall\,w\in H^{1}_{0}(\Omega).\nonumber
\end{eqnarray}
By Theorem \ref{estimates} we have $\|u_n\|_{H^1(\Omega)}\leq C,\
\|\varphi_n\|_{W^{1,r}(\Omega)}\leq C$ for all $n$, where $C>0$
denotes a generic constant independent of $n$. Therefore, on a
subsequence
\begin{eqnarray*}
u_n \stackrel{w}\rightharpoonup u^*\ {\rm in}\ H^1(\Omega)\ {\rm
and}\ \varphi_n\stackrel{w}\rightharpoonup \varphi^*\ {\rm in}\
W^{1,r}(\Omega).
\end{eqnarray*}
Also, $\beta_n\in U_{\cal M}$ for all $n$ implies that $\beta_n\in
L^\infty(\Gamma_R)$ and
$\|\beta_n\|_{L^\infty(\Gamma_R)}\leq {\cal M}$ for all $n$.
Hence, on a subsequence
$\beta_n\stackrel{w*}\rightharpoonup{\beta^*}$ in
$L^\infty(\Gamma_R)$.
%i.e.,
%\begin{eqnarray*}%\label{weakstar}
%\int_{\partial\Omega}\beta_n w\,ds\rightarrow\int
%_{\partial\Omega}\tilde\beta w\,ds\quad\forall\,w\in
%L^1(\partial\Omega)\ {\rm as}\ n\rightarrow\infty.
%\end{eqnarray*}
On the other hand, $\beta_n\in L^2(\Gamma_R)$ and
$\|\beta_n\|_{L^2(\Gamma_R)}\leq {\cal M}$ for each $n$ implies
$\beta_n\stackrel{w}\rightharpoonup\beta^*$ in $L^2(\Gamma_R)$.
%which means
%\begin{eqnarray*}%\label{weak}
%\int_{\partial\Omega}\beta_n v\,ds\rightarrow\int
%_{\partial\Omega}\beta^* v\,ds\quad\forall\,v\in
%L^2(\partial\Omega)\ {\rm as}\ n\rightarrow\infty.
%\end{eqnarray*}
%It is easy to see that $\tilde{\beta}=\beta^*$ a.e.
%Note that $\beta^*\in U_M$ because $U_M$ being a closed convex set
%of a Hilbert space $\cal H$ implies that $U_M$ is closed with
%respect to weak convergence. Indeed, let $v\in
%L^2(\partial\Omega)$ be arbitrary. Then by (\ref{weakstar}) and
%(\ref{weak})
%$$\Big|\int_{\partial\Omega}(\tilde\beta-\beta^*)v\,ds\Big|\leq
%\Big|\int_{\partial\Omega}(\tilde\beta-\beta_n)v\,ds\Big|+
%\Big|\int_{\partial\Omega}(\beta_n-\beta^*)v\,ds\Big|\rightarrow 0
%\text{ as } n\rightarrow\infty.$$ Thus
%$$\int_{\partial\Omega}(\tilde\beta-\beta^*)v\,ds=0\quad\forall\,v\in
%L^2(\partial\Omega)$$ whence $\tilde\beta=\beta^*$ a.e.

As mentioned in Section~\ref{criticaltemp}, we have that
$H^1(\Omega)\subset\subset L^s(\Omega)$, and since $r>d$,
$W^{1,r}(\Omega)\subset\subset C(\bar\Omega)$, where $r$ and $s$
were chosen in (\ref{r}). Hence, on a subsequence
\begin{eqnarray}
\varphi_n&\stackrel{s}\rightarrow &\varphi^*\ {\rm in}\
C(\bar\Omega),\
\nabla\varphi_n\stackrel{w}\rightharpoonup\nabla\varphi^*\
{\rm in}\ L^r(\Omega),\nonumber\\
u_n&\stackrel{s}\rightarrow & u^*\ {\rm in}\ L^s(\Omega),\ \nabla
u_n\stackrel{w}\rightharpoonup\nabla u^*\ {\rm in}\
L^2(\Omega),\label{allinall}\\
%u_n&\stackrel{s}\rightarrow u^*\text{ in $L^s(\Omega)$},\\
%\varphi_n&\stackrel{s}\rightarrow \varphi^*\text{ in
%$C(\bar\Omega)$},\\
\beta_n&\stackrel{w}\rightharpoonup&\beta^*\ {\rm in}\
L^2(\Gamma_R),\ \beta_n\stackrel{w*}\rightharpoonup\beta^*\
{\rm in}\ L^\infty(\Gamma_R).\nonumber
\end{eqnarray}
Now we need to show that $u^*=u(\beta^*)$ and
$\varphi^*=\varphi(\beta^*)$ solve (\ref{eq1}) with control
$\beta^*$, i.e., pass to the limit as $n\rightarrow\infty$ in
(\ref{minseq}).
%Notice that \eqref{conjugate} and
%\eqref{conjugate1} imply
%\begin{equation}
%\frac{1}{s}+\frac{1}{r}+\frac{1}{2}=1.
%\end{equation}
From (\ref{allinall}) it is immediate that
\begin{eqnarray*}
\int_\Omega\nabla u_n\cdot\nabla v\,dx\rightarrow
\int_\Omega\nabla u^*\cdot\nabla v\,dx\ {\rm as}\
n\rightarrow\infty.
\end{eqnarray*}
Using the trace inequality, the fact that
$H^1(\Omega)\subset\subset L^2(\partial\Omega)$ and
$\beta_n\stackrel{w*}\rightharpoonup\beta^*$ in
$L^\infty(\Gamma_R)$ we can show that
\begin{eqnarray}\label{betat}
\int_{\Gamma_R}\beta_n u_n v\,ds\rightarrow
\int_{\Gamma_R}\beta^* u^* v\,ds\ {\rm as}\
n\rightarrow\infty.
\end{eqnarray}
%Indeed, by the trace inequality $u^*\in H^1(\Omega)$ implies
%$u^*\in L^2(\partial\Omega)$. As $v\in L^2(\partial\Omega)$ it
%folows $u^*v\in L^1(\partial\Omega)$ and we get
%\begin{eqnarray*}
%\Big|\int_{\partial\Omega}\beta_n u_n v\,ds-
%\int_{\partial\Omega}\beta^* u^* v\,ds\Big|&\leq&
%\int_{\partial\Omega}|\beta_n u_n v-\beta_n u^* v|\,ds+
%\Big|\int_{\partial\Omega}\beta_n u^* v-\beta^* u^*
%v\,ds\Big|\\
%&\leq& M\int_{\partial\Omega}|u_n-u^*|\cdot|v|\,ds+
%\Big|\int_{\partial\Omega}(\beta_n -\beta^* )u^* v\,ds\Big| \\
%&\leq&
%M\|u_n-u^*\|_{L^2(\partial\Omega)}\cdot\|v\|_{L^2(\partial\Omega)}\\
%&+&\Big|\int_{\partial\Omega}(\beta_n -\beta^* )u^*
%v\,ds\Big|\rightarrow 0\ {\rm as}\ n\rightarrow\infty,
%\end{eqnarray*}
%where we have taken into account that $H^1(\Omega)\subset\subset
%L^2(\partial\Omega)$ and
%$\beta_n\stackrel{w*}\rightharpoonup\beta^*$ in
%$L^\infty(\partial\Omega)$. This proves (\ref{betat}).
Next we show that
\begin{eqnarray}\label{phi}
\int_\Omega\sigma(u_n)\,v\nabla\varphi_n\cdot\nabla\varphi{_0}\,dx
\rightarrow
\int_\Omega\sigma(u^*)\,v\nabla\varphi^*\cdot\nabla\varphi_{0}\,dx
\ {\rm as}\ n\rightarrow\infty.
\end{eqnarray}
We have
\begin{eqnarray*}
&&\Big|\int_\Omega\sigma(u_n)\,v\nabla\varphi_n\cdot\nabla
\varphi_{0}\,dx-
\int_\Omega\sigma(u^*)\,v\nabla\varphi^*\cdot\nabla\varphi_{0}\,dx
\Big|\\
&\leq&\Big|\int_\Omega[\sigma(u_n)-\sigma(u^*)]\,v\nabla\varphi_n\cdot\nabla
\varphi_{0}\,dx\Big|
+\Big|\int_\Omega\sigma(u^*)\,v\,(\nabla\varphi_n-\nabla\varphi^*)
\cdot\nabla\varphi_{0}\,dx\Big|\\
&\leq& K\|\nabla\varphi_0\|_{\infty}\int_\Omega
|u_n-u^*|\cdot|v|\cdot|\nabla\varphi_n|\,dx+\Big|\int_\Omega\sigma(u^*)\,v\,(\nabla\varphi_n-\nabla\varphi^*)
\cdot\nabla\varphi_{0}\,dx\Big|\\
&\leq& K\|\nabla\varphi_0\|_{\infty}
\|u_n-u^*\|_s\cdot\|v\|_2\cdot\|\nabla\varphi_n\|_r+\Big|\int_\Omega\sigma(u^*)\,v\,(\nabla\varphi_n-\nabla\varphi^*)
\cdot\nabla\varphi_{0}\,dx\Big|\rightarrow 0,%\ {\rm as}\
%n\rightarrow\infty,
\end{eqnarray*}
where $K$ is the Lipschitz constant for $\sigma$, and where we
took into account that $u_n\stackrel{s}\rightarrow u^*$ in
$L^s(\Omega)$,
$\nabla\varphi\stackrel{w}\rightharpoonup\nabla\varphi^*$ in
$L^2(\Omega)$, and $\|\nabla\varphi_n\|_r\leq C$. This completes
the proof of (\ref{phi}). Similarly, using the corresponding
convergences from (\ref{allinall}) we can show that
\begin{eqnarray*}\label{nablaphi}
\int_\Omega (\varphi_{0}-\varphi_n)\,\sigma(u_n)\nabla
\varphi_n\cdot\nabla v\,dx\rightarrow\int_\Omega
(\varphi_{0}-\varphi^*)\,\sigma(u^*)\nabla \varphi^*\cdot\nabla
v\,dx\\
\int_\Omega\sigma(u_n)\nabla \varphi_n\cdot\nabla
w\,dx\rightarrow\int_\Omega\sigma(u^*)\nabla \varphi^*\cdot\nabla
w\,dx\ {\rm as }\ n\rightarrow\infty.
\end{eqnarray*}
More details on the above convergences can be found in \cite{Hry}.
%Having verified the convergence of each term in (\ref{minseq}) and
Now letting $n\rightarrow\infty$ in (\ref{minseq}), we obtain
\begin{eqnarray*}
\int_\Omega\nabla u^*\cdot\nabla
v\,dx+\int_{\Gamma_R}\beta^* u^* v\,dx&=&\int_\Omega
(\varphi_{0}-\varphi^*)\sigma(u^*)\nabla \varphi^*\cdot\nabla
v\,dx\\
&+&\int_\Omega\sigma(u^*)\nabla
\varphi^*\cdot\nabla\varphi_{0}\,v\,dx\ \forall\,v\in H^{1}(\Omega) \\
\int_\Omega\sigma(u^*)\nabla \varphi^*\cdot\nabla w\,dx&=&0
\hskip0.2in \varphi_{0}-\varphi_n\in H_{0}^{1}(\Omega),
\forall\,w\in H^{1}_{0}(\Omega).
\end{eqnarray*}
Therefore $(u^*,\varphi^*)$ is a weak solution associated with
$\beta^*$: $u^*=u(\beta^*)$ and $\varphi^*=\varphi(\beta^*)$.

\noindent Now we show that $\beta^*$ is optimal. %As
%$u_n\stackrel{w}\rightharpoonup u^*$ in $L^2(\Omega)$, it follows
%that $\lim_{n\rightarrow\infty}\int_\Omega u_n\,dx$ exists and
%\begin{eqnarray*}
%\lim_{n\rightarrow\infty}\int_\Omega u_n\,dx=\int_\Omega u^*\,dx.
%\end{eqnarray*}
As $\lim_{n\rightarrow\infty}J(\beta_n)$ exists %and
%$\lim_{n\rightarrow\infty}\int_\Omega u_n\,dx$ exist,
we conclude $\lim_{n\rightarrow\infty}\int_{\Gamma_R}
\beta_n^2\,ds$ exists and
%\begin{eqnarray*}
%\int_{\partial\Omega} (\beta^*)^2\,ds\leq\lim_{n\rightarrow\infty}
%\int_{\partial\Omega}\beta_n^2\,ds.
%\end{eqnarray*}
%Therefore, we can write
\begin{eqnarray*}
\inf_{\beta\in U_M}J(\beta)&=&\lim_{n\rightarrow\infty}J(\beta_n)
=\lim_{n\rightarrow
\infty}\int_{\Omega}u_n\,dx+\lim_{n\rightarrow\infty}\int_{\Gamma_R}\beta_n^{2}\,ds\\
&\geq&\int_{\Omega}u^*\,dx+\int_{\Gamma_R}(\beta^{*})^{2}\,ds=J(\beta^*).
\end{eqnarray*}
This implies that $\beta^*$ is an optimal control.
\end{proof}

\section{Derivation of the optimality system}\label{optimality}

Our optimal control will be represented in terms of the solution
to the optimality system, which consists of the original state
system and the adjoint system whose construction, loosely
speaking, follows the standard technique of (i) deriving the
sensitivity equations (linearizing the state equations), (ii)
exchanging the role of test function and solution in the
linearization, and (iii) using the derivative of the cost function
with respect to the state as a non-homogeneity.

To obtain the necessary conditions for the optimality system we
differentiate the objective functional with respect to the
control. Since the objective functional depends on $u$, and $u$ is
coupled to $\varphi$ through a PDE, we will need to differentiate
$u$ and $\varphi$ with respect to the control $\beta$.
\begin{theorem}\label{sensitivity}
{\rm(Sensitivities)} If the boundary data $\varphi_{_0}$ are
sufficiently small, i.e., if
$\|\varphi_0\|_{W^{1,\infty}(\Omega)}$ is small enough, then the
mapping $\beta\mapsto(u,\varphi)$ is differentiable in the
following sense:
\begin{eqnarray}
\frac{u(\beta+\varepsilon\ell)-u(\beta)}{\varepsilon}&\stackrel{w}
\rightharpoonup&\psi_1\ {\rm in}\ H^1(\Omega),\nonumber\\
\label{psionetwo}\\[-1.5ex]
\frac{\varphi(\beta+\varepsilon\ell)-\varphi(\beta)}{\varepsilon}&\stackrel{w}
\rightharpoonup&\psi_2\ {\rm in}\ H^1_0(\Omega)\ {\rm as}\
\varepsilon\rightarrow 0 \nonumber
\end{eqnarray}
for any $\beta\in U_{\cal M}$ and $\ell\in
L^\infty(\partial\Omega)$ such that $(\beta+\varepsilon\ell)\in
U_{\cal M}$ for small $\varepsilon$. Moreover, the sensitivities,
$\psi_1\in H^1(\Omega)$ and $\psi_2\in H^1_0(\Omega)$, satisfy
\begin{eqnarray}
\Delta\psi_1+\sigma^{\prime}(u)|\nabla\varphi|^2\psi_1+2\sigma(u)
\nabla\varphi\cdot\nabla\psi_2&=&0\ {\rm in}\ \Omega,\nonumber\\
\nabla\cdot[\sigma^{\prime}(u)\psi_1\nabla\varphi+\sigma(u)\nabla\psi_2]&=&0
\ {\rm in}\ \Omega,\nonumber\\
\label{sensitivityeqn}\\[-3.7ex]
\frac{\partial\psi_1}{\partial n}+\beta\psi_1+\ell (u-u_1)&=&0\
{\rm on}\
\Gamma_R,\nonumber\\
\psi_1&=&0\ {\rm on}\ \Gamma_D,\nonumber\\
\psi_2&=&0\ {\rm on}\ \partial\Omega.\nonumber
\end{eqnarray}
\end{theorem}
\begin{proof}
We follow \cite{Hry} with appropriate modifications where
necessary. Earlier we denoted $u=u(\beta)$ and
$\varphi=\varphi(\beta)$. Denote also
$u^\varepsilon=u(\beta^\varepsilon)$,
$\varphi^\varepsilon=\varphi(\beta^\varepsilon)$, where
$\beta^\varepsilon\stackrel{\rm{def}}=\beta+\varepsilon\ell$. The
weak formulation for $(u^\varepsilon,\varphi^\varepsilon)$ is
\begin{eqnarray}
\int_\Omega\nabla u^\varepsilon\nabla
v\,dx+\int_{\Gamma_R}\beta^\varepsilon (u^\varepsilon-u_1)
v\,ds&=&\int_\Omega
(\varphi_{0}-\varphi^\varepsilon)\,\sigma(u^\varepsilon)\nabla
\varphi^\varepsilon\nabla
v\,dx \nonumber\\
&+&\int_\Omega\sigma(u^\varepsilon)v\nabla
\varphi^\varepsilon\nabla\varphi_{0}\,dx\quad \forall\,v\in
V_D(\Omega),\label{sensitivepsilon}\\[-1.ex]
\int_\Omega\sigma(u^\varepsilon)\nabla \varphi^\varepsilon\nabla
w\,dx&=&0\quad \varphi_{0}-\varphi^\varepsilon\in
H_{0}^{1}(\Omega), \forall\,w\in H^{1}_{0}(\Omega).\nonumber
\end{eqnarray}
Similarly, for $(u,\varphi)$ we have
\begin{eqnarray}
\int_\Omega\nabla u\cdot\nabla \tilde v\,dx+\int_{\Gamma_R}\beta
(u-u_1) \tilde v\,ds&=&\int_\Omega
(\varphi_{0}-\varphi)\,\sigma(u)\nabla \varphi\cdot\nabla
\tilde v\,dx \nonumber\\
&+&\int_\Omega\sigma(u)\nabla
\varphi\cdot\nabla\varphi_{0}\,\tilde v\,dx\quad \forall\,\tilde
v\in
V_D(\Omega),\label{sensitivestim}\\[-1.ex]
\int_\Omega\sigma(u)\nabla \varphi\cdot\nabla \tilde
w\,dx&=&0\quad \varphi-\varphi_{0}\in H_{0}^{1}(\Omega),
\forall\,\tilde w\in H^{1}_{0}(\Omega).\nonumber
\end{eqnarray}
Take the test functions $v=(u^\varepsilon-u)/{\varepsilon}$,
$\tilde v=(u^\varepsilon-u)/{\varepsilon}$,
$w=(\varphi^\varepsilon-\varphi)/{\varepsilon}$, and $\tilde
w=(\varphi^\varepsilon-\varphi)/{\varepsilon}$, subtract
corresponding equations in (\ref{sensitivestim}) from
(\ref{sensitivepsilon}), and divide by $\varepsilon$ to obtain
\begin{eqnarray}
&\int_\Omega\nabla
\Big(\frac{u^\varepsilon-u}{\varepsilon}\Big)\cdot\nabla
\Big(\frac{u^\varepsilon-u}{\varepsilon}\Big)\,dx+\int_{\Gamma_R}
\beta\Big(\frac{u^\varepsilon-u}{\varepsilon}\Big)\Big(\frac{u^\varepsilon-u}
{\varepsilon}\Big)\,ds\nonumber\\
&=-\int_{\Gamma_R}\ell (u^\varepsilon-u_1)
\Big(\frac{u^\varepsilon-u}{\varepsilon}\Big)\,ds\nonumber\\
&+\frac{1}{\varepsilon}\int_\Omega
\Big[(\varphi_{0}-\varphi^\varepsilon)\,\sigma(u^\varepsilon)
\nabla\varphi^\varepsilon-(\varphi_{0}-\varphi)\,\sigma(u)
\nabla\varphi\Big]\cdot\nabla\Big(\frac{u^\varepsilon-u}{\varepsilon}\Big)\,dx\nonumber\\
\label{sensitivdifference}\\[-2.3ex]
&+\frac{1}{\varepsilon}\int_\Omega\Big[\sigma(u^\varepsilon)\nabla
\varphi^\varepsilon-\sigma(u)\nabla
\varphi\Big]\cdot\nabla\varphi_{0}\Big(\frac{u^\varepsilon-u}{\varepsilon}\Big)\,dx,\nonumber\\
&\frac{1}{\varepsilon}\int_\Omega\Big[\sigma(u^\varepsilon)\nabla
\varphi^\varepsilon\cdot\nabla\Big(\frac{\varphi^\varepsilon-\varphi}{\varepsilon}\Big)
-\sigma(u)\nabla\varphi\cdot\nabla\Big(\frac{\varphi^\varepsilon-\varphi}{\varepsilon}
\Big)\Big]\,dx=0.\nonumber
\end{eqnarray}

%**************************************************************************************************

We derive $H^1(\Omega)$ estimate for
$(\varphi^\varepsilon-\varphi)/\varepsilon$ first. Since
$(\varphi^\varepsilon-\varphi)/\varepsilon\in H_0^1(\Omega)$ it
follows from the Poincar\'{e} inequality that it is sufficient to
have a bound on
$\|\nabla(\varphi^\varepsilon-\varphi)/\varepsilon\|_2$.

The second equation in \eqref{sensitivdifference} implies
\begin{equation}\label{phiepsilon}
\int_\Omega\sigma(u^\varepsilon)\nabla
\Big(\frac{\varphi^\varepsilon}{\varepsilon}\Big)\cdot\nabla\Big
(\frac{\varphi^\varepsilon-\varphi}{\varepsilon}\Big)\,dx
=\int_\Omega\sigma(u)\nabla
\Big(\frac{\varphi}{\varepsilon}\Big)\cdot\nabla\Big
(\frac{\varphi^\varepsilon-\varphi}{\varepsilon}\Big)\,dx.
\end{equation}
Taking into account \eqref{phiepsilon} we can write
\begin{equation}\label{phiestim}
\begin{split}
&\int_\Omega\sigma(u)\Big|\nabla\Big
(\frac{\varphi^\varepsilon-\varphi}{\varepsilon}\Big)\Big|^2\,dx=
\int_\Omega\sigma(u)\nabla\Big
(\frac{\varphi^\varepsilon-\varphi}{\varepsilon}\Big)\cdot\nabla\Big
(\frac{\varphi^\varepsilon-\varphi}{\varepsilon}\Big)\,dx\\
&=\int_\Omega\sigma(u)\nabla
\Big(\frac{\varphi^\varepsilon}{\varepsilon}\Big)\cdot\nabla\Big
(\frac{\varphi^\varepsilon-\varphi}{\varepsilon}\Big)\,dx-\int_\Omega\sigma(u)\nabla
\Big(\frac{\varphi}{\varepsilon}\Big)\cdot\nabla\Big
(\frac{\varphi^\varepsilon-\varphi}{\varepsilon}\Big)\,dx\\
&=\int_\Omega\sigma(u)\nabla
\Big(\frac{\varphi^\varepsilon}{\varepsilon}\Big)\cdot\nabla\Big
(\frac{\varphi^\varepsilon-\varphi}{\varepsilon}\Big)\,dx-
\int_\Omega\sigma(u^\varepsilon)\nabla\Big(\frac{\varphi^\varepsilon}
{\varepsilon}\Big)\cdot\nabla\Big(\frac{\varphi^\varepsilon-\varphi}{\varepsilon}\Big)\,dx\\
&=\int_\Omega\Big(\sigma(u)-\sigma(u^\varepsilon)\Big)\nabla
\Big(\frac{\varphi^\varepsilon}{\varepsilon}\Big)\cdot\nabla\Big
(\frac{\varphi^\varepsilon-\varphi}{\varepsilon}\Big)\,dx.
\end{split}
\end{equation}
{\bf Remark 1.} Observe that for a given weak solution $u$, it
follows from Theorem~\ref{existreg}, that if we set
$C_1(u)=\sigma(N)$ with $N:=\|u\|_{\infty}$, then (since $0\leq
u\leq N < u_{*}$) we will have $0<C_1(u)\leq\sigma(u)$. Thus, we
can write
\begin{equation}
\begin{split}
&C_1(u)\int_\Omega\Big|\nabla\Big
(\frac{\varphi^\varepsilon-\varphi}{\varepsilon}\Big)\Big|^2\,dx\leq
\int_\Omega\sigma(u)\Big|\nabla\Big
(\frac{\varphi^\varepsilon-\varphi}{\varepsilon}\Big)\Big|^2\,dx\\
&=\int_\Omega\Big(\sigma(u)-\sigma(u^\varepsilon)\Big)\nabla
\Big(\frac{\varphi^\varepsilon}{\varepsilon}\Big)\cdot\nabla\Big
(\frac{\varphi^\varepsilon-\varphi}{\varepsilon}\Big)\,dx\\
&\leq\int_\Omega\Big|\frac{\sigma(u)-\sigma(u^\varepsilon)}{\varepsilon}\Big|
\cdot|\nabla\varphi^\varepsilon|\cdot\Big|\nabla\Big(\frac{\varphi^\varepsilon-\varphi}
{\varepsilon}\Big)\Big|\,dx\\
&\leq K\int_\Omega\Big|\frac{u^\varepsilon-u}{\varepsilon}\Big|
\cdot|\nabla\varphi^\varepsilon|\cdot\Big|\nabla\Big(\frac{\varphi^\varepsilon-\varphi}
{\varepsilon}\Big)\Big|\,dx\\
&\leq
K\Big\|\frac{u^\varepsilon-u}{\varepsilon}\Big\|_s\cdot\|\nabla
\varphi^\varepsilon\|_r\cdot\Big\|\nabla\Big(\frac{\varphi^\varepsilon-\varphi}
{\varepsilon}\Big)\Big\|_2,
\end{split}
\end{equation}
where we used \eqref{phiestim}. Thus,
$C_1(u)\Big\|\nabla\Big(\frac{\varphi^\varepsilon-\varphi}
{\varepsilon}\Big)\Big\|_2^2\leq
K\Big\|\frac{u^\varepsilon-u}{\varepsilon}\Big\|_s\cdot\|\nabla
\varphi^\varepsilon\|_r\cdot\Big\|\nabla\Big(\frac{\varphi^\varepsilon-\varphi}
{\varepsilon}\Big)\Big\|_2$ and therefore
\begin{equation}\label{prelimsensitivity}
\Big\|\nabla\Big(\frac{\varphi^\varepsilon-\varphi}
{\varepsilon}\Big)\Big\|_2\leq
\frac{K}{C_1(u)}\Big\|\frac{u^\varepsilon-u}{\varepsilon}\Big\|_s\cdot\|\nabla
\varphi^\varepsilon\|_r\,.
\end{equation}
Since $H^1(\Omega)\subset L^s(\Omega)$ we have \vspace{-0.07in}
\begin{equation}\label{H1imbedding}
\Big\|\frac{u^\varepsilon-u}{\varepsilon}\Big\|_s\leq M_1
\Big\|\frac{u^\varepsilon-u}{\varepsilon}\Big\|_{H^1(\Omega)}.
\end{equation}
By Theorem \ref{estimates} we have
$\|\nabla\varphi^\varepsilon\|_r\leq\Phi$. Substituting this
estimate and \eqref{H1imbedding} into \eqref{prelimsensitivity} we
get
\begin{equation}
\Big\|\nabla\Big(\frac{\varphi^\varepsilon-\varphi}
{\varepsilon}\Big)\Big\|_2\leq
\frac{KM_1\Phi}{C_1(u)}\,\Big\|\frac{u^\varepsilon-u}{\varepsilon}\Big\|_{H^1(\Omega)}.
\end{equation}
Using the  Poincar\'{e} inequality, we obtain
\begin{equation}\label{keysensitestim}
\Big\|\frac{\varphi^\varepsilon-\varphi}
{\varepsilon}\Big\|_{H^1(\Omega)}\leq
C_6\Big\|\nabla\Big(\frac{\varphi^\varepsilon-\varphi}
{\varepsilon}\Big)\Big\|_2\leq
\frac{C_6KM_1\Phi}{C_1(u)}\,\Big\|\frac{u^\varepsilon-u}{\varepsilon}\Big\|_{H^1(\Omega)}.
\end{equation}
Now we proceed to estimate the $H^1$ norm of
$(u^\varepsilon-u)/\varepsilon$. We obtain from
\eqref{sensitivdifference}
\begin{equation*}
\begin{split}
&\int_\Omega\Big|\nabla
\Big(\frac{u^\varepsilon-u}{\varepsilon}\Big)\Big|^2\,dx\leq\int_\Omega\Big|\nabla
\Big(\frac{u^\varepsilon-u}{\varepsilon}\Big)\Big|^2\,dx+\int
_{\Gamma_R}\beta\Big(\frac{u^\varepsilon-u}{\varepsilon}
\Big)^2\,ds\\
%&=\Big|-\int_{\Gamma_R}\ell
%(u^\varepsilon-u_1)\Big(\frac{u^\varepsilon-u}{\varepsilon}
%\Big)\,ds+ \frac{1}{\varepsilon}\int_\Omega
%\Big[(\varphi_{_0}-\varphi^\varepsilon)\,\sigma(u^\varepsilon)
%\nabla\varphi^\varepsilon-(\varphi_{_0}-\varphi)\,\sigma(u)
%\nabla\varphi\Big]\nabla\Big(\frac{u^\varepsilon-u}{\varepsilon}\Big)\,dx\\
%&+\frac{1}{\varepsilon}\int_\Omega\Big[\sigma(u^\varepsilon)\nabla
%\varphi^\varepsilon-\sigma(u)\nabla
%\varphi\Big]\cdot\nabla\varphi_{_0}\Big(\frac{u^\varepsilon-u}{\varepsilon}\Big)\,dx\Big|\\
&=\Big|-\int_{\Gamma_R}\ell(u^\varepsilon-u_1)\Big(\frac{u^\varepsilon-u}{\varepsilon}
\Big)\,ds+{\cal C}+{\cal D}\Big|,
\end{split}
\end{equation*}
where ${\cal
C}\stackrel{\text{def}}=\frac{1}{\varepsilon}\int_\Omega
\Big[(\varphi_{_0}-\varphi^\varepsilon)\,\sigma(u^\varepsilon)
\nabla\varphi^\varepsilon-(\varphi_{_0}-\varphi)\,\sigma(u)
\nabla\varphi\Big]\cdot\nabla\Big(\frac{u^\varepsilon-u}{\varepsilon}\Big)\,dx
%&=\int_\Omega
%\Big[\frac{(\varphi_{_0}-\varphi+\varphi-\varphi^\varepsilon)\,\sigma(u^\varepsilon)
%\nabla\varphi^\varepsilon-(\varphi_{_0}-\varphi)\,\sigma(u)
%\nabla\varphi}{\varepsilon}\Big]\cdot\nabla\Big(\frac{u^\varepsilon-u}{\varepsilon}\Big)\,dx\\
=\int_\Omega
\Big[(\varphi_{_0}-\varphi)\Big(\frac{\sigma(u^\varepsilon)\nabla\varphi^\varepsilon-
\sigma(u)\nabla\varphi}
{\varepsilon}\Big)+\Big(\frac{\varphi-\varphi^\varepsilon}{\varepsilon}\Big)\sigma(u^\varepsilon
)\nabla\varphi^\varepsilon\Big]\cdot\nabla\Big(\frac{u^\varepsilon-u}{\varepsilon}\Big)\,dx$
and ${\cal
D}\stackrel{\text{def}}=\frac{1}{\varepsilon}\int_\Omega\Big[\sigma(u^\varepsilon)\nabla
\varphi^\varepsilon-\sigma(u)\nabla
\varphi\Big]\cdot\nabla\varphi_{_0}\Big(\frac{u^\varepsilon-u}{\varepsilon}\Big)\,dx.$
%\begin{equation*}
%\begin{split}
%{\cal C}&\stackrel{\text{def}}=\frac{1}{\varepsilon}\int_\Omega
%\Big[(\varphi_{_0}-\varphi^\varepsilon)\,\sigma(u^\varepsilon)
%\nabla\varphi^\varepsilon-(\varphi_{_0}-\varphi)\,\sigma(u)
%\nabla\varphi\Big]\cdot\nabla\Big(\frac{u^\varepsilon-u}{\varepsilon}\Big)\,dx\\
%%&=\int_\Omega
%%\Big[\frac{(\varphi_{_0}-\varphi+\varphi-\varphi^\varepsilon)\,\sigma(u^\varepsilon)
%%\nabla\varphi^\varepsilon-(\varphi_{_0}-\varphi)\,\sigma(u)
%%\nabla\varphi}{\varepsilon}\Big]\cdot\nabla\Big(\frac{u^\varepsilon-u}{\varepsilon}\Big)\,dx\\
%&=\int_\Omega
%\Big[(\varphi_{_0}-\varphi)\Big(\frac{\sigma(u^\varepsilon)\nabla\varphi^\varepsilon-
%\sigma(u)\nabla\varphi}
%{\varepsilon}\Big)+\Big(\frac{\varphi-\varphi^\varepsilon}{\varepsilon}\Big)\sigma(u^\varepsilon
%)\nabla\varphi^\varepsilon\Big]\cdot\nabla\Big(\frac{u^\varepsilon-u}{\varepsilon}\Big)\,dx,
%\ {\rm and}\\
%{\cal
%D}&\stackrel{\text{def}}=\frac{1}{\varepsilon}\int_\Omega\Big[\sigma(u^\varepsilon)\nabla
%\varphi^\varepsilon-\sigma(u)\nabla
%\varphi\Big]\cdot\nabla\varphi_{_0}\Big(\frac{u^\varepsilon-u}{\varepsilon}\Big)\,dx.
%\end{split}
%\end{equation*}
First, we estimate $|\cal C|$ to get
\begin{equation*}
\begin{split}
|\cal C|&\leq\int_\Omega
|\varphi_{_0}-\varphi|\cdot\Big|\frac{\sigma(u^\varepsilon)\nabla\varphi^\varepsilon-
\sigma(u)\nabla\varphi}
{\varepsilon}\Big|\cdot\Big|\nabla\Big(\frac{u^\varepsilon-u}{\varepsilon}\Big)\Big|\,dx\\
&+\int_\Omega\Big|\frac{\varphi-\varphi^\varepsilon}{\varepsilon}\Big|\,\sigma(u^\varepsilon
)|\nabla\varphi^\varepsilon|\cdot\Big|\nabla\Big(\frac{u^\varepsilon-u}{\varepsilon}\Big)\Big|\,dx\\
&\leq 2\tilde M\int_\Omega
\Big|\frac{\sigma(u^\varepsilon)\nabla\varphi^\varepsilon-
\sigma(u)\nabla\varphi}
{\varepsilon}\Big|\cdot\Big|\nabla\Big(\frac{u^\varepsilon-u}{\varepsilon}\Big)\Big|\,dx\\
&+
\mu\int_\Omega\Big|\frac{\varphi-\varphi^\varepsilon}{\varepsilon}\Big|\cdot
|\nabla\varphi^\varepsilon|\cdot\Big|\nabla\Big(\frac{u^\varepsilon-u}{\varepsilon}\Big)\Big|\,dx\\
&=2\tilde M\int_\Omega
\Big|\frac{(\sigma(u^\varepsilon)-\sigma(u)+\sigma(u))\nabla\varphi^\varepsilon-
\sigma(u)\nabla\varphi}
{\varepsilon}\Big|\cdot\Big|\nabla\Big(\frac{u^\varepsilon-u}{\varepsilon}\Big)\Big|\,dx\\
&+
\mu\int_\Omega\Big|\frac{\varphi-\varphi^\varepsilon}{\varepsilon}\Big|\cdot
|\nabla\varphi^\varepsilon|\cdot\Big|\nabla\Big(\frac{u^\varepsilon-u}{\varepsilon}\Big)\Big|\,dx\\
&\leq 2\tilde
M\int_\Omega\Big|\frac{\sigma(u^\varepsilon)-\sigma(u)}{\varepsilon}\Big|
\cdot
|\nabla\varphi^\varepsilon|\cdot\Big|\nabla\Big(\frac{u^\varepsilon-u}{\varepsilon}\Big)\Big|\,dx\\
&+2\tilde
M\int_\Omega\sigma(u)\cdot\Big|\nabla\Big(\frac{\varphi^\varepsilon-\varphi}{\varepsilon}\Big)\Big|
\cdot\Big|\nabla\Big(\frac{u^\varepsilon-u}{\varepsilon}\Big)\Big|\,dx\\
&+\mu\int_\Omega\Big|\frac{\varphi-\varphi^\varepsilon}{\varepsilon}\Big|\cdot
|\nabla\varphi^\varepsilon|\cdot\Big|\nabla\Big(\frac{u^\varepsilon-u}{\varepsilon}\Big)\Big|\,dx\\
&\leq
2\tilde{M}K\int_\Omega\Big|\frac{u^\varepsilon-u}{\varepsilon}\Big|
\cdot
|\nabla\varphi^\varepsilon|\cdot\Big|\nabla\Big(\frac{u^\varepsilon-u}{\varepsilon}\Big)\Big|\,dx\\
&+2\tilde M
\mu\int_\Omega\Big|\nabla\Big(\frac{\varphi^\varepsilon-\varphi}{\varepsilon}\Big)\Big|
\cdot\Big|\nabla\Big(\frac{u^\varepsilon-u}{\varepsilon}\Big)\Big|\,dx\\
&+\mu\int_\Omega\Big|\frac{\varphi-\varphi^\varepsilon}{\varepsilon}\Big|\cdot
|\nabla\varphi^\varepsilon|\cdot\Big|\nabla\Big(\frac{u^\varepsilon-u}{\varepsilon}\Big)\Big|\,dx.
\end{split}
\end{equation*}
Next, estimating $|\cal D|$ we obtain 
\begin{equation*}
\begin{split}
|{\cal D}|& \leq
%\int_\Omega|\nabla\varphi_{_0}|\cdot
%\Big|\frac{(\sigma(u^\varepsilon)-\sigma(u)+\sigma(u))\nabla\varphi^\varepsilon-\sigma(u)\nabla
%\varphi}{\varepsilon}\Big|\cdot\Big|\frac{u^\varepsilon-u}{\varepsilon}\Big|\,dx\\
\int_\Omega|\nabla\varphi_{_0}|\cdot
\Big|\frac{\sigma(u^\varepsilon)-\sigma(u)}{\varepsilon}\Big|\cdot|\nabla\varphi^\varepsilon|\cdot\Big|
\frac{u^\varepsilon-u}{\varepsilon}\Big|\,dx\\
&+\int_\Omega\sigma(u)|\nabla\varphi_{_0}|\cdot\Big|\nabla\Big(\frac{\varphi^\varepsilon-
\varphi}{\varepsilon}\Big)\Big|\cdot\Big|\frac{u^\varepsilon-u}{\varepsilon}\Big|\,dx\\
&\leq K\|\nabla\varphi_{_0}\|_\infty\int_\Omega
\Big|\frac{u^\varepsilon-u}{\varepsilon}\Big|\cdot|\nabla\varphi^\varepsilon|\cdot\Big|
\frac{u^\varepsilon-u}{\varepsilon}\Big|\,dx\\
&+\mu\|\nabla\varphi_{_0}\|_\infty\int_\Omega\Big|\nabla\Big(\frac{\varphi^\varepsilon-
\varphi}{\varepsilon}\Big)\Big|\cdot\Big|\frac{u^\varepsilon-u}{\varepsilon}\Big|\,dx.
\end{split}
\end{equation*}

Thus we have

\begin{equation*}
\begin{split}
&\int_\Omega\Big|\nabla
\Big(\frac{u^\varepsilon-u}{\varepsilon}\Big)\Big|^2\,dx\leq\int_{\Gamma_R}|\ell|\cdot|u^\varepsilon-u_1|\cdot\Big|
\frac{u^\varepsilon-u}{\varepsilon}\Big|\,ds+|{\cal C}|+|{\cal
D}|\\
&\leq {\cal M}\int_{\Gamma_R}|u^\varepsilon-u_1|\cdot\Big|
\frac{u^\varepsilon-u}{\varepsilon}\Big|\,ds+|{\cal C}|+|{\cal
D}|\\
& \leq {\cal M}
\|u^\varepsilon-u_1\|_{L^2(\Gamma_R)}\cdot\Big\|\frac{u^\varepsilon-u}{\varepsilon}
\Big\|_{L^2(\partial\Omega)}+|{\cal C}|+|{\cal
D}|\\
&\leq {\cal M} M_2^2\|u^\varepsilon-u_1\|_{H^1(\Omega)}\cdot\Big
\|\frac{u^\varepsilon-u}{\varepsilon}\Big\|_{H^1(\Omega)}+|{\cal
C}|+|{\cal D}|,
\end{split}
\end{equation*}
where we have used the trace inequality
$\|u^\varepsilon-u_1\|_{L^2(\partial\Omega)}\leq
M_2\|u^\varepsilon\|_{H^1(\Omega)}$. Taking into account the
estimates for $|\cal C|$ and $|\cal D|$ we get
\begin{equation*}
\begin{split}
&\int_\Omega\Big|\nabla
\Big(\frac{u^\varepsilon-u}{\varepsilon}\Big)\Big|^2\,dx %+\lambda\int_{\partial\Omega}\Big(\frac{u^\varepsilon-u}{\varepsilon}\Big)^2\,ds
\leq {\cal{M}}M_2^2\|u^\varepsilon-u_1\|_{H^1(\Omega)}\cdot\Big
\|\frac{u^\varepsilon-u}{\varepsilon}\Big\|_{H^1(\Omega)}\\
&+2\tilde{M}K\int_\Omega\Big|\frac{u^\varepsilon-u}{\varepsilon}\Big|
\cdot
|\nabla\varphi^\varepsilon|\cdot\Big|\nabla\Big(\frac{u^\varepsilon-u}{\varepsilon}\Big)\Big|\,dx+2\tilde
M
\mu\int_\Omega\Big|\nabla\Big(\frac{\varphi^\varepsilon-\varphi}{\varepsilon}\Big)\Big|
\cdot\Big|\nabla\Big(\frac{u^\varepsilon-u}{\varepsilon}\Big)\Big|\,dx\\
&+\mu\int_\Omega\Big|\frac{\varphi-\varphi^\varepsilon}{\varepsilon}\Big|\cdot
|\nabla\varphi^\varepsilon|\cdot\Big|\nabla\Big(\frac{u^\varepsilon-u}{\varepsilon}\Big)\Big|\,dx+
K\|\nabla\varphi_{_0}\|_\infty\int_\Omega
\Big|\frac{u^\varepsilon-u}{\varepsilon}\Big|\cdot|\nabla\varphi^\varepsilon|\cdot\Big|
\frac{u^\varepsilon-u}{\varepsilon}\Big|\,dx\\
&+\mu\|\nabla\varphi_{_0}\|_\infty\int_\Omega\Big|\nabla\Big(\frac{\varphi^\varepsilon-
\varphi}{\varepsilon}\Big)\Big|\cdot\Big|\frac{u^\varepsilon-u}{\varepsilon}\Big|\,dx.
\end{split}
\end{equation*}
Using the H\"{o}lder inequality, and {\it a priori} bounds
\eqref{ubound}, \eqref{H1imbedding}, and \eqref{keysensitestim} we
have
\begin{equation*}
\begin{split}
%\Big\|\frac{u^\varepsilon-u}{\varepsilon}\Big\|_*^2\equiv
&\int_\Omega\Big|\nabla
\Big(\frac{u^\varepsilon-u}{\varepsilon}\Big)\Big|^2\,dx
\\%+\lambda\int
%_{\partial\Omega}\Big(\frac{u^\varepsilon-u}{\varepsilon}
%\Big)^2\,ds\\
&\leq {\cal M}M_2^2\tilde C\Big
\|\frac{u^\varepsilon-u}{\varepsilon}\Big\|_{H^1(\Omega)}\\
&+2\tilde MK\Big
\|\frac{u^\varepsilon-u}{\varepsilon}\Big\|_s\cdot\|\nabla\varphi^\varepsilon\|_r\cdot
\Big
\|\nabla\Big(\frac{u^\varepsilon-u}{\varepsilon}\Big)\Big\|_2\\
&+2\tilde M\mu\Big
\|\nabla\Big(\frac{\varphi^\varepsilon-\varphi}{\varepsilon}\Big)\Big\|_2\cdot\Big
\|\nabla\Big(\frac{u^\varepsilon-u}{\varepsilon}\Big)\Big\|_2\\
&+\mu
\Big\|\frac{\varphi^\varepsilon-\varphi}{\varepsilon}\Big\|_s\cdot
\|\nabla\varphi^\varepsilon\|_r\cdot\Big
\|\nabla\Big(\frac{u^\varepsilon-u}{\varepsilon}\Big)\Big\|_2\\
&+K\|\nabla\varphi_{_0}\|_\infty
\cdot\Big\|\frac{u^\varepsilon-u}{\varepsilon}\Big\|_s\cdot\|\nabla\varphi^\varepsilon\|_r\cdot
\Big\|\frac{u^\varepsilon-u}{\varepsilon}\Big\|_2\\
&+\mu\|\nabla\varphi_{_0}\|_\infty\cdot\Big
\|\nabla\Big(\frac{\varphi^\varepsilon-\varphi}{\varepsilon}\Big)\Big\|_2\cdot
\Big\|\frac{u^\varepsilon-u}{\varepsilon}\Big\|_2 \\
\end{split}
\end{equation*}
\begin{equation*}
\begin{split}
&\leq {\cal
M}M_2^2\tilde{C}\Big\|\frac{u^\varepsilon-u}{\varepsilon}\Big\|_{H^1(\Omega)}+2\tilde
MKM_1\Phi\Big\|\frac{u^\varepsilon-u}{\varepsilon}\Big\|_{H^1(\Omega)}^2\\
&+\frac{2\tilde M \mu
KM_1\Phi}{C_1(u)}\Big\|\frac{u^\varepsilon-u}{\varepsilon}\Big\|_{H^1(\Omega)}^2
+\frac{\mu C_6KM_1\Phi^2}{C_1(u)}\Big\|\frac{u^\varepsilon-u}{\varepsilon}\Big\|_{H^1(\Omega)}^2\\
&+K\|\nabla\varphi_{_0}\|_\infty
M_1\Phi\Big\|\frac{u^\varepsilon-u}{\varepsilon}\Big\|_{H^1(\Omega)}^2+\mu\|\nabla\varphi_{_0}\|_\infty
\frac{KM_1\Phi}{C_1(u)}\Big\|\frac{u^\varepsilon-u}{\varepsilon}\Big\|_{H^1(\Omega)}^2.
\end{split}
\end{equation*}
Since $u-u^{\varepsilon}\in V_D$, it follows from the extension of
the Poincar\'{e} inequality (see Theorem 5.8 in \cite{Rod}) that
there exists $k>0$ such that
\begin{equation*}
\begin{split}
&k\Big\|\frac{u^\varepsilon-u}{\varepsilon}\Big\|_{H^1(\Omega)}^2\leq
%\Big\|\frac{u^\varepsilon-u}{\varepsilon}\Big\|_*^2
\int_\Omega\Big|\nabla
\Big(\frac{u^\varepsilon-u}{\varepsilon}\Big)\Big|^2\,dx\leq
{\cal M}M_2^2\tilde{C}\Big\|\frac{u^\varepsilon-u}{\varepsilon}\Big\|_{H^1(\Omega)}\\
&+2\tilde
MKM_1\Phi\Big\|\frac{u^\varepsilon-u}{\varepsilon}\Big\|_{H^1(\Omega)}^2+\frac{2\tilde
M \mu
KM_1\Phi}{C_1(u)}\Big\|\frac{u^\varepsilon-u}{\varepsilon}\Big\|_{H^1(\Omega)}^2
+\frac{\mu C_6KM_1\Phi^2}{C_1(u)}\Big\|\frac{u^\varepsilon-u}{\varepsilon}\Big\|_{H^1(\Omega)}^2\\
&+K\|\nabla\varphi_{_0}\|_\infty
M_1\Phi\Big\|\frac{u^\varepsilon-u}{\varepsilon}\Big\|_{H^1(\Omega)}^2+\mu\|\nabla\varphi_{_0}\|_\infty
\frac{KM_1\Phi}{C_1(u)}\Big\|\frac{u^\varepsilon-u}{\varepsilon}\Big\|_{H^1(\Omega)}^2.
\end{split}
\end{equation*}
By definition, $\Phi$ includes $\|\nabla\varphi_0\|_\infty$ and
$\|\varphi_0\|_{W^{1,\infty}(\Omega)}$. %(see \eqref{eq6a}).
Hence,  if $\|\varphi_0\|_{W^{1,\infty}(\Omega)}$ is chosen small
enough so that
\begin{equation}
\begin{split}
k_1&\equiv k-2\tilde MKM_1\Phi-\frac{2\tilde M
\mu KM_1\Phi}{C_1(u)}-\frac{\mu C_6KM_1\Phi^2}{C_1(u)}\\
&-K\|\nabla\varphi_{_0}\|_\infty
M_1\Phi-\mu\|\nabla\varphi_{_0}\|_\infty \frac{KM_1\Phi}{C_1(u)}>0
\end{split}
\end{equation}
then
\begin{equation}\label{usensitivbound}
\Big\|\frac{u^\varepsilon-u}{\varepsilon}\Big\|_{H^1(\Omega)}\leq
\frac{{\cal M}M_2^2\tilde{C}}{k_1}\,,
\end{equation}
where the constant in \eqref{usensitivbound} does not depend on
$\varepsilon$. Consequently, \eqref{keysensitestim} yields
\begin{equation}\label{phisensitivbound}
\Big\|\frac{\varphi^\varepsilon-\varphi}
{\varepsilon}\Big\|_{H^1(\Omega)}\leq
\frac{C_6KM_1\Phi}{C_1(u)}\,\Big\|\frac{u^\varepsilon-u}{\varepsilon}\Big\|_{H^1(\Omega)}
\leq \frac{C_6KM_1\Phi}{C_1(u)}\cdot\frac{{\cal
M}M_2^2\tilde{C}}{k_1}\,.
\end{equation}

These estimates justify the existence of $\psi_1$ and $\psi_2$,
and the convergences in (\ref{psionetwo}). All in all, we have the
following convergences
\begin{eqnarray*}
\frac{u^\varepsilon-u}{\varepsilon}&\stackrel{w}\rightharpoonup&\psi_1\
{\rm in}\ H^1(\Omega),\
\frac{u^\varepsilon-u}{\varepsilon}\stackrel{s}\rightarrow\psi_1\
{\rm in}\ L^s(\Omega),\\
\frac{\varphi^\varepsilon-\varphi}
{\varepsilon}&\stackrel{w}\rightharpoonup&\psi_2\ {\rm in}\
H^1(\Omega),\ \frac{\varphi^\varepsilon-\varphi}
{\varepsilon}\stackrel{s}\rightarrow\psi_2\ {\rm
in}\ L^s(\Omega),\\
u^\varepsilon&\stackrel{s}\rightarrow& u\ {\rm in}\ L^s(\Omega),\
\varphi^\varepsilon\stackrel{s}\rightarrow\varphi\ {\rm in}\
C(\bar\Omega),\\
%\nabla u^\epsilon&\stackrel{w}\rightharpoonup&\nabla u\ {\rm in}
%L^2(\Omega),\\
\nabla\varphi^\varepsilon&\stackrel{w}\rightharpoonup&\nabla\varphi\
{\rm in}\ L^r(\Omega),\
\frac{u^\varepsilon-u}{\varepsilon}\stackrel{w}\rightharpoonup\psi_1\
{\rm in}\ L^2(\Gamma_R),\\
\beta^\varepsilon&\stackrel{w}\rightharpoonup&\beta\ {\rm in}\
L^2(\Gamma_R),\
\beta^\varepsilon\stackrel{w*}\rightharpoonup\beta\ {\rm in}\
L^\infty(\Gamma_R)\ {\rm as}\ \varepsilon\rightarrow 0,
\end{eqnarray*}
and therefore, we can show that the sensitivities satisfy the
system (\ref{sensitivityeqn}).

\medskip

\noindent{\bf Remark 2.} At this point the convergences are on a
subsequence and in order to obtain the desired convergence for the
whole sequence, it needs to be shown that the limits $\psi_1$ and
$\psi_2$ are always the same for any subsequence. This uniqueness
of the limits follows from \eqref{sensitivityeqn} (since the
system in $\psi_1$ and $\psi_2$ is non-degenerate linear elliptic)
which $\psi_1$ and $\psi_2$ produced by subsequences will
necessarily satisfy.

Subtracting (\ref{sensitivestim}) from (\ref{sensitivepsilon}),
and dividing by $\varepsilon$:
\begin{eqnarray}\label{sensitivdifference1}
&&\int_\Omega\nabla
\Big(\frac{u^\varepsilon-u}{\varepsilon}\Big)\cdot\nabla
v\,dx+\int_{\Gamma_R}\beta
\Big(\frac{u^\varepsilon-u}{\varepsilon}\Big)v\,ds+\int_{\Gamma_R}
\ell (u^\varepsilon-u_1)v\,ds\nonumber\\
&=&\frac{1}{\varepsilon}\int_\Omega
\Big[(\varphi_{_0}-\varphi^\varepsilon)\,\sigma(u^\varepsilon)
\nabla\varphi^\varepsilon\cdot\nabla
v-(\varphi_{_0}-\varphi)\,\sigma(u)
\nabla\varphi\cdot\nabla v\Big]\,dx\\
&+&\frac{1}{\varepsilon}\int_\Omega\Big[\sigma(u^\varepsilon)\nabla
\varphi^\varepsilon-\sigma(u)\nabla
\varphi\Big]\cdot\nabla\varphi_{_0}v\,dx \quad\forall\,v\in
V_D(\Omega),\nonumber
\end{eqnarray}
\begin{eqnarray}\label{sensitivdifference2}
&\frac{1}{\varepsilon}\int_\Omega\Big[\sigma(u^\varepsilon)\nabla
\varphi^\varepsilon\cdot\nabla w
-\sigma(u)\nabla\varphi\cdot\nabla w\Big]\,dx=0\quad\forall\,w\in
H_0^1(\Omega).
\end{eqnarray}
The convergence proofs for various terms are rather standard if sometimes
lengthy. For example, it can be shown easily that the terms on the
left hand side of (\ref{sensitivdifference1}) converge because of the
weak convergence of the corresponding sequences. We omit the details.
%It is immediate that for the left hand side of
%\eqref{sensitivdifference1} we get
%\begin{equation*}
%\begin{split}
%\int_\Omega\nabla
%\Big(\frac{u^\varepsilon-u}{\varepsilon}\Big)\cdot\nabla
%v\,dx&\rightarrow\int_\Omega\nabla\psi_1\cdot\nabla
%v\,dx,\\
%\int_{\partial\Omega}\beta
%\Big(\frac{u^\varepsilon-u}{\varepsilon}\Big)v\,ds&\rightarrow
%\int_{\partial\Omega}\beta \psi_1v\,ds,\text{ and}\\
%\int_{\partial\Omega}\ell u^\varepsilon
%v\,ds&\rightarrow\int_{\partial\Omega}\ell u v\,ds\quad\text{as
%$\varepsilon\rightarrow 0$}.
%\end{split}
%\end{equation*}
The first term on the right hand side of
(\ref{sensitivdifference1}) can be written as follows
\begin{eqnarray*}
&&\frac{1}{\varepsilon}\int_\Omega
\Big[(\varphi_{_0}-\varphi^\varepsilon)\,\sigma(u^\varepsilon)
\nabla\varphi^\varepsilon\nabla
v-(\varphi_{_0}-\varphi)\,\sigma(u) \nabla\varphi\nabla
v\Big]\,dx\\
%&=\frac{1}{\varepsilon}\int_\Omega
%\Big[(\varphi_{_0}-\varphi^\varepsilon)\,\sigma(u^\varepsilon)
%\nabla\varphi^\varepsilon-(\varphi_{_0}-\varphi^\varepsilon+\varphi^\varepsilon-\varphi)\,\sigma(u)
%\nabla\varphi\Big]\cdot\nabla v\,dx\\
&=&\frac{1}{\varepsilon}\int_\Omega
(\varphi_{_0}-\varphi^\varepsilon)\Big[\sigma(u^\varepsilon)\nabla\varphi^\varepsilon-
\sigma(u)\nabla\varphi\Big]\nabla
v\,dx+\frac{1}{\varepsilon}\int_\Omega
(\varphi-\varphi^\varepsilon)\sigma(u)\nabla\varphi\nabla
v\,dx\\
&=&\frac{1}{\varepsilon}\int_\Omega(\varphi_{_0}-\varphi^\varepsilon)
\Big[\sigma(u^\varepsilon)-\sigma(u)\Big]
\nabla\varphi^\varepsilon\nabla v\,dx
+\frac{1}{\varepsilon}\int_\Omega(\varphi_{_0}-\varphi^\varepsilon)
\,\sigma(u)\Big[\nabla\varphi^\varepsilon-\nabla\varphi\Big]\nabla
v\,dx\\
&+&\frac{1}{\varepsilon}\int_\Omega
(\varphi-\varphi^\varepsilon)\sigma(u)\nabla\varphi\nabla
v\,dx={\cal G}_1+{\cal G}_2+{\cal
F}.\\
\end{eqnarray*}
%We illustrate the type of estimates by considering terms coming
%from ${\cal G}$. We write
%%where
%%\begin{equation*}
%%\begin{split}
%%&{\cal G}\stackrel{\text{def}}=\frac{1}{\varepsilon}\int_\Omega
%%(\varphi_{_0}-\varphi^\varepsilon)\Big[\sigma(u^\varepsilon)\nabla\varphi^\varepsilon-
%%\sigma(u)\nabla\varphi\Big]\cdot\nabla v\,dx\\ &{\cal
%%F}\stackrel{\text{def}}=\frac{1}{\varepsilon}\int_\Omega
%%(\varphi-\varphi^\varepsilon)\sigma(u)\nabla\varphi\cdot\nabla
%%v\,dx.
%%\end{split}
%%\end{equation*}
%%First, we show
\noindent A detailed derivation for the convergence of ${\cal
G}_1$ can be found in \cite{Hry}. Therefore, for the sake of
completeness, we show the convergence of ${\cal G}_2$ and ${\cal
F}$ here. First, we need to show that
\begin{equation}\label{Fconverg}
{\cal F}\rightarrow-\int_\Omega
\psi_2\,\sigma(u)\nabla\varphi\nabla v\,dx\text{ as
$\varepsilon\rightarrow 0$}.
\end{equation}
Indeed, we can write
\begin{equation*}
\begin{split}
&\Big|\int_\Omega
\Big(\frac{\varphi-\varphi^\varepsilon}{\varepsilon}\Big)\sigma(u)\nabla\varphi\nabla
v\,dx-\int_\Omega(-\psi_2)\sigma(u)\nabla\varphi\nabla
v\,dx\Big|\\
&\leq
C_2\int_\Omega\Big|\frac{\varphi^\varepsilon-\varphi}{\varepsilon}-\psi_2\Big||\varphi||\nabla
v|\,dx\leq
C_2\Big\|\frac{\varphi^\varepsilon-\varphi}{\varepsilon}-\psi_2\Big\|_s\|\nabla\varphi\|_r\|\nabla
v\|_2\rightarrow 0\text{ as $\varepsilon\rightarrow 0$}.
\end{split}
\end{equation*}
This proves \eqref{Fconverg}.
%\begin{eqnarray*}
%{\cal G}&=&\frac{1}{\varepsilon}\int_\Omega
%(\varphi_{_0}-\varphi^\varepsilon)\Big[\sigma(u^\varepsilon)\nabla\varphi^\varepsilon-
%\sigma(u)\nabla\varphi\Big]\cdot\nabla
%v\,dx\\
%%&=\frac{1}{\varepsilon}\int_\Omega
%%(\varphi_{_0}-\varphi^\varepsilon)\Big[(\sigma(u^\varepsilon)-\sigma(u)+\sigma(u))
%%\nabla\varphi^\varepsilon-\sigma(u)\nabla\varphi\Big]\cdot\nabla
%%v\,dx\\
%&=&\frac{1}{\varepsilon}\int_\Omega(\varphi_{_0}-\varphi^\varepsilon)
%\Big[\sigma(u^\varepsilon)-\sigma(u)\Big]
%\nabla\varphi^\varepsilon\cdot\nabla v\,dx\\
%&+&\frac{1}{\varepsilon}\int_\Omega(\varphi_{_0}-\varphi^\varepsilon)
%\,\sigma(u)\Big[\nabla\varphi^\varepsilon-\nabla\varphi\Big]\cdot\nabla
%v\,dx={\cal G}_1+{\cal G}_2
%\end{eqnarray*}
%where
%\begin{equation*}
%\begin{split}
%&{\cal
%G}_1\stackrel{\text{def}}=\frac{1}{\varepsilon}\int_\Omega(\varphi_{0}-\varphi^\varepsilon)
%\Big[\sigma(u^\varepsilon)-\sigma(u)\Big]
%\nabla\varphi^\varepsilon\cdot\nabla v\,dx,\\
%&{\cal
%G}_2\stackrel{\text{def}}=\frac{1}{\varepsilon}\int_\Omega(\varphi_{0}-\varphi^\varepsilon)
%\,\sigma(u)\Big[\nabla\varphi^\varepsilon-\nabla\varphi\Big]\cdot\nabla
%v\,dx.
%\end{split}
%\end{equation*}
Next, we show that
\begin{eqnarray}\label{G2}
{\cal G}_2\rightarrow\int_\Omega(\varphi_{0}-\varphi)
\,\sigma(u)\nabla\psi_2\nabla v\,dx\ {\rm as}\
\varepsilon\rightarrow 0.
\end{eqnarray}
%and we illustrate terms from ${\cal G}_1$. Namely, we show
We write
\begin{equation}
\begin{split}
{\cal G}_2
%&=\frac{1}{\varepsilon}\int_\Omega(\varphi_{0}-\varphi+\varphi-\varphi^\varepsilon)
%\,\sigma(u)\Big[\nabla\varphi^\varepsilon-\nabla\varphi\Big]\cdot\nabla
%v\,dx\\
&=\frac{1}{\varepsilon}\int_\Omega(\varphi_{0}-\varphi)
\,\sigma(u)\nabla(\varphi^\varepsilon-\varphi)\nabla v\,dx\\
&+\frac{1}{\varepsilon}\int_\Omega(\varphi-\varphi^\varepsilon)
\,\sigma(u)\nabla(\varphi^\varepsilon-\varphi)\nabla v\,dx={\cal
G}_{21}+{\cal G}_{22}.
\end{split}
\end{equation}
%where
%\begin{equation*}
%\begin{split}
%&{\cal
%G}_{21}\stackrel{\text{def}}=\frac{1}{\varepsilon}\int_\Omega(\varphi_{0}-\varphi)
%\,\sigma(u)\nabla(\varphi^\varepsilon-\varphi)\nabla v\,dx\\
%&{\cal
%G}_{22}\stackrel{\text{def}}=\frac{1}{\varepsilon}\int_\Omega(\varphi-\varphi^\varepsilon)
%\,\sigma(u)\nabla(\varphi^\varepsilon-\varphi)\nabla v\,dx.
%\end{split}
%\end{equation*}
%We show
%\begin{align}
%&{\cal G}_{21}\rightarrow\int_\Omega(\varphi_{0}-\varphi)
%\,\sigma(u)\nabla\psi_2\cdot\nabla
%v\,dx\label{G21},\\
%&{\cal G}_{22}\rightarrow 0\label{G22}\quad\text{as
%$\varepsilon\rightarrow 0$}.
%\end{align}
%Indeed, for \eqref{G21} we can write
For the term ${\cal G}_{21}$ we have
\begin{equation*}
\begin{split}
&\Big|{\cal G}_{21}-\int_\Omega(\varphi_{0}-\varphi)
\,\sigma(u)\nabla\psi_2\cdot\nabla v\,dx\Big|\\
&=\Big|\int_\Omega\Big[(\varphi_{0}-\varphi)
\,\sigma(u)\nabla\Big(\frac{\varphi^\varepsilon-\varphi}{\varepsilon}\Big)\cdot\nabla
v-(\varphi_{0}-\varphi) \,\sigma(u)\nabla\psi_2\cdot\nabla
v\Big]\,dx\\
&=\Big|\int_\Omega\Big[\nabla\Big(\frac{\varphi^\varepsilon-\varphi}{\varepsilon}\Big)-
\nabla\psi_2\Big](\varphi_{0}-\varphi)\,\sigma(u)\nabla
v\,dx\Big|\rightarrow 0\quad\text{as $\varepsilon\rightarrow 0$}
\end{split}
\end{equation*}
since
$\nabla(\varphi^\varepsilon-\varphi)/\varepsilon\stackrel{w}\rightharpoonup\nabla\psi_2$
in $L^2(\Omega)$ and $(\varphi_{0}-\varphi)\,\sigma(u)\nabla v\in
L^2(\Omega)$ because $r>d$ implies $\varphi_{0}-\varphi\in
W^{1,r}(\Omega)\subset\subset C(\bar\Omega)$. Now we show the
convergence of ${\cal G}_{22}$. Indeed, we have
\begin{equation*}
\begin{split}
|{\cal G}_{22}|&=\Big|\int_\Omega(\varphi-\varphi^\varepsilon)
\,\sigma(u)\nabla\Big(\frac{\varphi^\varepsilon-\varphi}{\varepsilon}\Big)\cdot\nabla
v\,dx \Big|\\
&\leq C_2\int_\Omega|\varphi-\varphi^\varepsilon|\cdot
\Big|\nabla\Big(\frac{\varphi^\varepsilon-\varphi}{\varepsilon}\Big)\Big|\cdot|\nabla
v|\,dx\\
&\leq
C_2\|\varphi-\varphi^\varepsilon\|_{C(\bar\Omega)}\cdot\Big\|\nabla
\Big(\frac{\varphi^\varepsilon-\varphi}{\varepsilon}\Big)
\Big\|_2\cdot\|\nabla v\|_2\rightarrow 0\quad\text{as
$\varepsilon\rightarrow 0$}
\end{split}
\end{equation*}
since $\|\varphi-\varphi^\varepsilon\|_{C(\bar\Omega)}\rightarrow
0$ as $\varepsilon\rightarrow 0$, and $\|\nabla
(\varphi^\varepsilon-\varphi)/\varepsilon)\|_2$ is bounded. This
ends the proof of \eqref{G2}.
Similarly, it can be shown for the second term on the right hand
side of (\ref{sensitivdifference1}) that:
\begin{eqnarray*}%\label{lastterm}
\frac{1}{\varepsilon}\int_\Omega\Big[\sigma(u^\varepsilon)\nabla
\varphi^\varepsilon-\sigma(u)\nabla
\varphi\Big]\cdot\nabla\varphi_{0}v\,dx&\rightarrow&\int_\Omega
\sigma^\prime(u)\psi_1\nabla\varphi\cdot\nabla\varphi_{0}v\,dx\\
&+&\int_\Omega \sigma(u)\nabla\psi_2\cdot\nabla\varphi_{0}v\,dx\
{\rm as}\ \varepsilon\rightarrow 0.
\end{eqnarray*}
Finally, the terms of (\ref{sensitivdifference2}) satisfy
\begin{eqnarray*}%\label{sensitivityphi}
\frac{1}{\varepsilon}\int_\Omega\Big[\sigma(u^\varepsilon)\nabla
\varphi^\varepsilon\cdot\nabla w
-\sigma(u)\nabla\varphi\cdot\nabla
w\Big]\,dx&\rightarrow&\int_\Omega\sigma^\prime(u)\psi_1\nabla
\varphi\cdot\nabla w\,dx\\
&+&\int_\Omega\sigma(u)\nabla\psi_2\cdot\nabla w\,dx\ {\rm as}\
\varepsilon\rightarrow 0.
\end{eqnarray*}
Letting $\varepsilon\rightarrow 0$ in (\ref{sensitivdifference1})
and (\ref{sensitivdifference2}) we obtain
\begin{eqnarray}
&&\int_\Omega\nabla \psi_1\nabla
v\,dx+\int_{\Gamma_R}(\beta\psi_1+\ell (u-u_1))
v\,ds\nonumber=\int_\Omega
(\varphi_{_0}-\varphi)\,\sigma^\prime(u)\psi_1
\nabla\varphi\nabla v\,dx\nonumber\\
&+&\int_\Omega (\varphi_{_0}-\varphi)\,\sigma(u)\nabla\psi_2
\nabla v\,dx
-\int_\Omega\psi_2\,\sigma(u)\nabla\varphi\nabla v\,dx\nonumber\\
\label{sensitweakformul}\\[-3.2ex]
&+&\int_\Omega\sigma^\prime(u)\psi_1\nabla\varphi\nabla\varphi_{_0}v\,dx
+\int_\Omega\sigma(u)\nabla\psi_2\nabla\varphi_{_0}v\,dx\quad\forall\,v\in V_D(\Omega), \nonumber\\
&&\int_\Omega\sigma^\prime(u)\psi_1\nabla \varphi\nabla w
+\int_\Omega\sigma(u)\nabla\psi_2\nabla w\,dx=0\quad \forall\,w\in
H_0^1(\Omega).\nonumber
\end{eqnarray}
%To determine the ``strong" formulation corresponding to
%(\ref{sensitweakformul}) we differentiate (\ref{statet}) with
%respect to $\beta$.  %(remembering that in a G{\^{a}}teaux sense
%%$({\partial u}/{\partial \beta})=\psi_1$, $({\partial
%%\varphi}/{\partial \beta})=\psi_2$, and $({\partial
%%\sigma(u)}/{\partial \beta})=\sigma^\prime(u)\psi_1$)
It can shown that the ``strong" formulation corresponding to
(\ref{sensitweakformul}) is given by
\begin{eqnarray}
\Delta\psi_1+\sigma^\prime(u)\psi_1|\nabla\varphi|^2+2\sigma(u)\nabla\varphi\cdot\nabla\psi_2&=&0
\ {\rm in}\ \Omega,\nonumber\\
\nabla\cdot(\sigma^\prime(u)\psi_1\nabla\varphi+\sigma(u)\nabla\psi_2)&=&0\
{\rm in}\ \Omega,\nonumber\\
\label{differentstate}\\[-2.9ex]
\frac{\partial\psi_1}{\partial n}+\beta\psi_1+\ell (u-u_1)&=&0\
{\rm on}\ \Gamma_R,\nonumber\\
\psi_1&=&0\ {\rm on}\ \Gamma_D,\nonumber\\
\psi_2&=&0\ {\rm on}\ \partial\Omega.\nonumber
\end{eqnarray}
\end{proof}

In order to characterize the optimal control, we need to introduce
adjoint functions $p$ and $q$ as well as the adjoint operator
associated with $\psi_1$ and $\psi_2$. Using the same reasoning as
in \cite{Hry}, it can be shown that the adjoint system is given by
\begin{eqnarray}
\Delta p+\sigma^{\prime}(u)|\nabla\varphi|^2p-\sigma^\prime(u)
\nabla\varphi\cdot\nabla q&=&1\ {\rm in}\ \Omega,\nonumber\\
\nabla\cdot[-2p\,\sigma(u)\nabla\varphi+\sigma(u)\nabla
q]&=&0\ {\rm in}\ \Omega,\nonumber\\
\label{adjoint1}\\[-3.9ex]
\frac{\partial p}{\partial n}+\beta^* p&=&0\ {\rm on}\
\Gamma_{R},\nonumber\\
p&=&0\ {\rm on}\ \Gamma_D,\nonumber\\
q&=&0\ {\rm on}\ \partial\Omega,\nonumber
\end{eqnarray}
where the nonhomogeneous term ``1" comes from differentiating the
integrand of $J(\beta)$ with respect to the state $u$.
\begin{theorem}\label{adjoint_theorem}
Let $\|\varphi_0 \|_{W^{1,\infty}(\Omega)}$ be sufficiently small.
Then, given an optimal control $\beta^*\in U_{\cal M}$ and the
corresponding states $u,\varphi$, there exists a solution
$(p,q)\in H^1(\Omega)\times H_0^1(\Omega)$ to the adjoint system
(\ref{adjoint1}).
%\begin{equation}\label{adj}
%\begin{split}
%&\Delta p+\sigma^{\prime}(u)|\nabla\varphi|^2p-\sigma^\prime(u)
%\nabla\varphi\cdot\nabla q=1\text{ in $\Omega$},\\
%&\nabla\cdot[-2p\,\sigma(u)\nabla\varphi+\sigma(u)\nabla
%q]=0\text{ in $\Omega$},\\
%&\frac{\partial p}{\partial n}+\beta^* p=0\text{ on
%$\partial\Omega$},\\
%&q=0\text{ on $\partial\Omega$}.
%\end{split}
%\end{equation}
Furthermore, $\beta^*$ can be explicitly characterized as:
\begin{equation}\label{characterization2}
\beta^*(x)=\min\Bigg(\max\Big(-\frac{(u-u_1)p}{2},0\Big),{\cal
M}\Bigg).
\end{equation}
\end{theorem}
\begin{proof}
Observe that the existence of solution to the adjoint system
(\ref{adjoint1}) can be proved using Banach fixed point theorem in
a similar way to how it was done in \cite{Hry}. Now we consider
the derivation of the characterization of the optimal control. For
a variation $\ell\in L^{\infty}(\Gamma_R)$, with
$\beta^*+\varepsilon\ell\in U_{\cal M}$, the weak formulation of
the sensitivity system \eqref{sensitivityeqn} is given by
\begin{eqnarray}
&&-\int_\Omega\nabla\psi_1\cdot\nabla
v\,dx-\int_{\Gamma_R}\beta\psi_1
v\,ds+\int_\Omega\sigma^\prime(u)\psi_1|\nabla\varphi|^2v\,dx\nonumber\\
&+&2\int_\Omega \sigma(u)\nabla\varphi\cdot\nabla\psi_2 v\,dx
=\int_{\Gamma_R}\ell (u-u_1)
v\,ds\nonumber \quad\forall\,v\in V_D(\Omega),\nonumber\\
\label{weakf1}\\[-1.ex]
&&\int_\Omega(\sigma^\prime(u)\psi_1\nabla\varphi+\sigma(u)\nabla\psi_2)\cdot\nabla
w\,dx=0\quad \forall\,w\in H_0^1(\Omega).\nonumber%\label{weakf2}
\end{eqnarray}
Since the minimum of $J$ is achieved at $\beta^*$ and for small
$\varepsilon>0$, $\beta^*+\varepsilon\ell\in U_{\cal M}$, we
obtain
\begin{eqnarray*}
0&\leq&\lim_{\varepsilon\rightarrow
0^+}\frac{J(\beta^*+\varepsilon\ell)-J(\beta^*)}{\varepsilon}\\
%&= \lim_{\varepsilon\rightarrow
%0^+}\frac{1}{\varepsilon}\Bigg[\int_\Omega
%u^\varepsilon\,dx+\int_{\partial\Omega}(\beta^*+\varepsilon
%\ell)^2\,ds-\int_\Omega
%u^*\,dx-\int_{\partial\Omega}(\beta^*)^2\,ds\Bigg]\\
%&=\lim_{\varepsilon\rightarrow 0^+}\Bigg[\int_\Omega
%\frac{u^\varepsilon-u^*}{\varepsilon}\,dx+\int_{\partial\Omega}
%\frac{(\beta^*+\varepsilon\ell)^2-(\beta^*)^2}{\varepsilon}\,ds\Bigg]\\
%&= \lim_{\varepsilon\rightarrow 0^+}\int_\Omega
%\frac{u^\varepsilon-u^*}{\varepsilon}\,dx+\lim_{\varepsilon\rightarrow
%0^+}\int_{\partial\Omega}
%\frac{(\beta^*)^2+2\beta^*\varepsilon\ell+\varepsilon^2\ell^2-(\beta^*)^2}
%{\varepsilon}\,ds\\
&=&\int_\Omega\psi_1\,dx+\int_{\Gamma_R}2\beta^*\ell\,ds
=\int_\Omega \left(\begin{array}{cc} \psi_1&\psi_2
\end{array}\right)
\left(\begin{array}{c}
1\\
0
\end{array}\right)
\,dx+\int_{\Gamma_R}2\beta^*\ell\,ds\\
&=&-\int_\Omega\nabla p\nabla \psi_1\,dx-\int_{\Gamma_R}\beta^*
p\,\psi_1\,ds+\int_\Omega\psi_1\,\sigma^{\prime}(u)|\nabla\varphi|^2p\,dx
-\int_\Omega\psi_1\,\sigma^\prime(u)\nabla\varphi\nabla
q\,dx\\
&+&2\int_\Omega p\,\sigma(u)\nabla\varphi\cdot\nabla
\psi_2\,dx-\int_\Omega\sigma(u)\nabla q\cdot\nabla\psi_2\,dx
+\int_{\Gamma_R}2\beta^*\ell\,ds\\
&=&\Big\{-\int_\Omega\nabla p\nabla
\psi_1\,dx-\int_{\Gamma_R}\beta^*
p\,\psi_1\,ds+\int_\Omega\psi_1\,\sigma^{\prime}(u)|\nabla\varphi|^2p\,dx\\
&+&2\int_\Omega p\,\sigma(u)\nabla\varphi\nabla\psi_2\,dx \Big\}
+\Big[-\int_\Omega\psi_1\,\sigma^\prime(u)\nabla\varphi\cdot\nabla
q\,dx-\int_\Omega\sigma(u)\nabla
q\cdot\nabla\psi_2\,dx\Big]\\
&+&\int_{\Gamma_R}2\beta^*\ell\,ds=\int_{\Gamma_R}\ell (u-u_1)
p\,ds+\int_{\Gamma_R}2\beta^*\ell\,ds,
\end{eqnarray*}
where we integrated by parts and used (\ref{weakf1}) with test
functions $p$ and $q$. Hence,
\begin{equation}\label{variation}
\int_{\Gamma_R}\ell(2\beta^*+(u-u_1)p)\,ds\geq 0.
\end{equation}
By the same argument as in \cite{Hry} one can obtain the explicit
characterization (\ref{characterization2}) of the optimal control.
Namely,
\begin{enumerate}[(i)]
\item Taking the variation $\ell$ to have support on the set
$\{x\in\Gamma_R:\lambda<\beta^*(x)<{\cal M}\}$ implies that the
variation $\ell(x)$ can be of any sign, and therefore we obtain
$2\beta^*+(u-u_1)p=0$ which leads to
$\beta^*=-\frac{(u-u_1)p}{2}\,$.

\item On the set $\{x\in\Gamma_R:\beta^*(x)={\cal M}\}$, the
variation must satisfy $\ell(x)\leq 0$ and therefore we get
$2\beta^*+(u-u_1)p\leq 0$ implying ${\cal
M}=\beta^*(x)\leq-\frac{(u-u_1)p}{2}\,$.

\item On the set $\{x\in\Gamma_R:\beta^*(x)=\lambda\}$, the
variation must satisfy $\ell(x)\geq 0$. This implies
$2\beta^*+(u-u_1)p\leq 0$ and thus
$\lambda=\beta^*(x)\geq-\frac{(u-u_1)p}{2}\,$.
\end{enumerate}

 Combining cases
(i), (ii), and (iii) gives the explicit characterization
(\ref{characterization2}) of the optimal control $\beta$.
%\begin{equation}\label{characterization}
%\beta^*(x)=
%\begin{cases}
%-\frac{up}{2},\text{ if $\lambda<-\frac{up}{2}<M$},\\
%M,\text{ if $-\frac{up}{2}\geq M$},\\
%\lambda,\text{ if $-\frac{up}{2}\leq\lambda$}.
%\end{cases}
%\end{equation}
%We can write \eqref{characterization} compactly as
%\begin{equation}\label{characterization1}
%\beta^*(x)=\min\Bigg(\max\Big(-\frac{up}{2},\lambda\Big),M\Bigg).%\mbox{\qedhere}
%\end{equation}
\qquad
\end{proof}

Substituting (\ref{characterization2}) into the state system
(\ref{statet}) and the adjoint equations (\ref{adjoint1}) we
obtain the optimality system:
\begin{eqnarray}
\Delta u^*+{\sigma}(u^*)|\nabla\varphi^*|^{2}&=&0\ {\rm in}\ \Omega,\nonumber\\
\nabla\cdot(\sigma(u^*)\nabla\varphi^*)&=&0\ {\rm in}\ \Omega,\nonumber\\
\Delta
p+\sigma^{\prime}(u^*)|\nabla\varphi^*|^2p-\sigma^\prime(u^*)
\nabla\varphi^*\cdot\nabla q&=&1\ {\rm in}\ \Omega,\nonumber\\
\nabla\cdot[-2p\,\sigma(u^*)\nabla\varphi^*+\sigma(u^*)\nabla
q]&=&0\ {\rm in}\ \Omega,\nonumber\\
\label{OS}\\[-3.5ex]
\frac{\partial p}{\partial n}+\min(\max(-{(u^*-u_1)p}/{2},0),{\cal
M})p&=&0\ {\rm on}\
\ \Gamma_R,\nonumber\\
p&=&0\ {\rm on}\ \Gamma_D,\nonumber\\
q&=&0\ {\rm on}\ \partial\Omega,\nonumber\\
{\displaystyle{\frac{\partial u^*}{\partial
n}}}+\min(\max(-{(u^*-u_1)p}/{2},0),{\cal M})(u^*-u_1)&=&0\ {\rm
on}\ \Gamma_R,\nonumber\\
u&=&u_0\ {\rm on}\ \Gamma_D,\nonumber\\
\varphi^*&=&\varphi_0\ {\rm on}\ \partial\Omega.\nonumber
\end{eqnarray}
Note that existence of solution to the optimality system
(\ref{OS}) follows from the existence of solution to the state
system (\ref{statet}) and Theorem \ref{adjoint_theorem}. %Also,

\end{document}